\documentclass[a4paper]{article}

\usepackage[english]{babel}

\usepackage[utf8]{inputenc}
\usepackage[utf8]{inputenc}
\usepackage{amsmath}
\usepackage{graphicx}
\usepackage{amssymb}
\usepackage{amsthm}
\usepackage{tikz-cd}
\usepackage{mathrsfs}
\usepackage[colorinlistoftodos]{todonotes}
\usepackage{enumitem}
\usepackage{yfonts}
\usepackage{ dsfont }
\usepackage{MnSymbol}
\usepackage{slashed}

\title{Iterated collapsing phenomenon on $G_2$-manifolds}

\author{Yang Li}

\date{\today}
\newtheorem{thm}{Theorem}[section]
\newtheorem{lem}[thm]{Lemma}

\theoremstyle{definition}
\newtheorem{eg}[thm]{Example}

\newtheorem{rmk}[thm]{Remark}
\newtheorem{prop}[thm]{Proposition}

\newtheorem*{Acknowledgement}{Acknowledgement}

\newcommand{\ie}{\emph{i.e.} }
\newcommand{\cf}{\emph{cf.} }

\newcommand{\R}{\mathbb{R}}
\newcommand{\C}{\mathbb{C}}
\newcommand{\Z}{\mathbb{Z}}

\newcommand{\norm}[1]{\left\lVert#1\right\rVert}
\newcommand{\Lap}{\Delta}

\DeclareMathOperator{\Tr}{Tr}

\def\XXint#1#2#3{{\setbox0=\hbox{$#1{#2#3}{\int}$ }
		\vcenter{\hbox{$#2#3$ }}\kern-.6\wd0}}

\begin{document}
	\maketitle

\begin{abstract}
We propose a new collapsing mechanism for $G_2$-metrics, with the generic region admitting a circle bundle structure over a K3 fibration over a Riemann surface. The adiabatic description involves a weighted version of the maximal submanifold equation. In a local smooth setting we prove the existence of formal power series solutions, and the problem of compactification is discussed at a heuristic level.
\end{abstract}

\section{Introduction}

The purpose of this paper is to introduce a formal differential geometric mechanism for the degeneration of $G_2$-metrics, which we hope will evantually lead to the construction of new $G_2$-metrics. In the generic region the manifold is approximately a circle fibration over a 6-dimensional manifold/orbifold which itself admits a K3 fibration over a Riemann surface $S$ (with boundary), where the sizes of the circle and the K3 fibres are shrinking at a fine tuned rate. This collapsing mechanism exhibits a plethora of phenomena:

\begin{itemize}
\item It can be viewed as a degenerate case of circle collapsing over a Calabi-Yau 3-fold studied by Foscolo-Haskins-Nordstr\"om \cite{HaskinsS1}. However, unlike in \cite{HaskinsS1}, by tuning the collapsing rate of the circle and K3 fibres, our setup maintains some nonlinearity of the Apostolov-Salamon equation, namely the $S^1$-reduction of the $G_2$-holonomy condition.

\item
It has close analogy with Donaldson's proposal of adiabatic Kovalev-Lefschetz fibration over a 3-dimensional base \cite{Donaldson}. Unlike in \cite{Donaldson}, the new feature is that K3 collapsing can happen without being the fibres of a coassociative fibration.

\item 
It generalizes in $G_2$ geometry some aspects of the `small complex structure limit' of Calabi-Yau metrics studied by Sun and Zhang \cite{SunZhang}. In particular, the Gibbons-Hawking ansatz and the Ooguri-Vafa type metrics in \cite{SunZhang} play a crucial role here in describing wall crossing behaviour.

\item
It involves fibrations of small ALF gravitational instantons along some 3-dimensional locus which itself admits a fibration over $S^1$ with small Riemann surface fibres. The picture has close analogy with ALE fibrations discussed in the physics literature, and the gluing construction of Joyce and Karigiannis \cite{JoyceKarigiannis}.

\end{itemize}

Our central philosophy, much like in Donaldson's proposal \cite{Donaldson}, is to encode the degenerating $G_2$-structures into certain adiabatic data. Roughly speaking, the topology is encoded into a local system over the Riemann surface $S$ with fibres isomorphic to $H^2(K3)$, and the $G_2$-structure is encoded by a section of this local system satisfying a number of linear algebraic constraints coming from the topology of the $S^1$-fibration. The $G_2$-holonomy condition then requires this section to be the critical point of a weighted area functional, analogous to the maximal submanifold equation proposed in \cite{Donaldson}. This local picture is put on solid foundation, as we demonstrate how to reconstruct a formal power series solution in the collapsing parameter $\epsilon$, starting from the data of a weighted maximal submanifold (\cf section \ref{Formalpowerseriessolution}). This is accomplished by a quite intricate induction scheme, after taking into account several gauge fixing issues.

 Global questions will be discussed at a heuristic level, such as Lefschetz fibration, wall crossing phenomena, boundary behaviour of the Riemann surface, ALF fibrations, etc. The walls indicate certain jumping discontinuities, and the positions of the walls are not known a priori, but should instead be solved along with the PDE, so the proposed global weighted maximal submanifold equation will have the nature of a free boundary problem.

The following problems still need to be resolved to turn our proposal into actual constructions of $G_2$-metrics on compact manifolds:

\begin{itemize}
\item Find the topological data over the Riemann surface $S$, using the global Torelli theorem for K3 surfaces and lattice theory. This problem is similar to the `matching problem' in the Kovalev twisted sum construction \cite{Kovalev}\cite{HaskinsCorti}.

\item
Solve the free boundary problem to obtain the adiabatic data, prove the regularity of solutions, and check certain genericity assumptions.

\item
Perform the gluing construction. This will require substantial work, and our more limited goal here is to discuss some likely geometric ingredients.

\end{itemize}

The ansatz in the $G_2$-setting also has a natural 
$Spin(7)$ analogue: a small circle bundle over a 7-manifold with a closed $G_2$-structure admitting a collapsing K3 fibration. It turns out to be encoded into a 3-dimensional version of the weighted maximal submanifold equation. This provides a unifying viewpoint for circle collapsing of $Spin(7)$ manifolds, and the adiabatic coassociative K3 fibration proposed by Donaldson \cite{Donaldson}. In particular, the $Spin(7)$ version of our ansatz generalizes the Donaldson ansatz.

\begin{Acknowledgement}
The author is a 2020 Clay Research Fellow, based at MIT. This work was partially done at the IAS. The author thanks Simon Donaldson, Song Sun and Mark Haskins for discussions.

\end{Acknowledgement}

\section{Local differential geometry}

\subsection{Circle collapsing and Apostolov-Salamon equations}

We briefly review the construction in \cite{HaskinsS1}.
A $G_2$-structure $(M,\phi)$ with a free $S^1$-symmetry action can be described in terms of an $SU(3)$-structure $(\omega,\Omega, g_{M/S^1})$ on $M/S^1$, together with a function $h: M/S^1\to \R_+$ encoding the length of the Killing vector field, and an $S^1$-connection 1-form $\vartheta$ on the principal circle bundle $M\to M/S^1$. Explicitly 
the $SU(3)$-structure is specified by a  non-degenerate real 2-form $\omega$ and a complex volume form $\Omega$ satisfying
\[
\omega^3= \frac{3}{2} \text{Re}\Omega\wedge \text{Im}\Omega, \quad \omega\wedge \Omega=0.
\]
and the
 $G_2$-structure is 
\[
\phi= \vartheta \wedge \omega+ h^{3/4}\text{Re}\Omega, \quad *_\phi\phi= -h^{1/4} \vartheta \wedge \text{Im}\Omega + \frac{1}{2}h\omega^2, \quad g_\phi= h^{1/2}g_{M/S^1} +h^{-1}\vartheta^2.
\]
The torsion free condition for the $G_2$-structure is equivalent to the \textbf{Apostolov-Salamon equation}
\begin{equation*}
\begin{cases}
d\omega=0, \quad d(h^{3/4}\text{Re}\Omega)=- d\vartheta \wedge \omega, \\
d(h^{1/4}\text{Im}\Omega)=0, \quad \frac{1}{2}dh\wedge \omega^2= h^{1/4} d\vartheta \wedge \text{Im}\Omega.
\end{cases}
\end{equation*}
In particular $[d\vartheta]\wedge [\omega]=0\in H^4(M)$, where $\frac{1}{2\pi}[d\vartheta]$ represents the first Chern class of the principal circle bundle.

While this is in general a highly nonlinear coupled system, the key observation of Foscolo-Haskins-Nordstr\"om \cite{HaskinsS1} is that it formally linearizes in the adiabatic limit where the $S^1$-fibres are much smaller compared to the size of $M/S^1$. They consider a family of $S^1$-invariant torsion free $G_2$-structures, written in a rescaled convention
\begin{equation}\label{ASrescaledconvention}
\phi_\epsilon= \epsilon \vartheta \wedge \omega + h^{3/4} \text{Re}\Omega,  \quad g_{\phi_\epsilon}= h^{1/2}g_{M/S^1} +\epsilon^2 h^{-1}\vartheta^2,
\end{equation}
where the data $(\omega, \Omega, h, \vartheta)$ depend on $\epsilon$. The torsion free condition then reads
\begin{equation}\label{AScollapsing}
\begin{cases}
d\omega=0, \quad \frac{1}{2} dh\wedge \omega^2= \epsilon h^{1/4} d\vartheta \wedge \text{Im}\Omega,
\\
d\text{Re}\Omega = -\frac{3}{4} h^{-1} dh\wedge \text{Re}\Omega - \epsilon h^{-3/4} d\vartheta \wedge \omega,
\quad d\text{Im}\Omega= - \frac{1}{4} h^{-1} dh \wedge \text{Im}\Omega.
\end{cases}
\end{equation}
Combining with the $SU(3)$-structure condition, we see
\[
d\vartheta\wedge \omega^2= -\epsilon^{-1} d (h^{3/4}\text{Re}\Omega \wedge \omega )=0.
\]
Assuming $(\omega, \Omega, h, \vartheta)$ all have smooth limits as $\epsilon\to 0$, we find that $h$ is constant to leading order, which we can normalize to be one. Then to leading order
\[
\omega=\omega_{CY}+ O(\epsilon), \quad \Omega=\Omega_{CY}+O(\epsilon),\quad 
d\omega_{CY}=0, \quad d\Omega_{CY}=0,
\] 
namely the $SU(3)$-structure is approximately \textbf{Calabi-Yau}. Writing \[h=1+ \epsilon \mathfrak{h}+ O(\epsilon^2),
\quad \vartheta= \vartheta_0+O(\epsilon), \]
we obtain the linear limiting  equations 
\begin{equation}
d\vartheta_0 \wedge \omega_{CY}^2=0, \quad \frac{1}{2}d\mathfrak{h} \wedge \omega_{CY}^2= d\vartheta_0\wedge \text{Im}\Omega_{CY}.
\end{equation}
Equivalently $(\mathfrak{h}, \vartheta_0)$ satisfies the \textbf{Calabi-Yau monopole} equation
\[
d\vartheta_0 \wedge \omega_{CY}^2=0, \quad d\mathfrak{h} = *(d\vartheta_0\wedge \text{Re}\Omega_{CY}).
\]
In particular $d^*d \mathfrak{h}=0$.
In the special case $\mathfrak{h}=0$, this reduces to the Hermitian Yang-Mills condition on the $U(1)$-connection $\vartheta_0$. Pluging this back into (\ref{AScollapsing}) has the effect of deforming the $SU(3)$-structure to being only approximately Calabi-Yau. Assuming there are no singular fibres, and under the cohomological condition $[d\vartheta]\wedge [\omega]=0$, then one can solve (\ref{AScollapsing}) iteratively to obtain a convergent expansion in $\epsilon$, and thereby produce infinitely many families of collapsing $G_2$-metrics on noncompact manifolds \cite{HaskinsS1}.

On compact manifolds there are no nontrivial examples of $G_2$-manifolds with $S^1$-symmetry, because any Killing vector field on a compact Ricci flat manifold is parallel. A folklore construction strategy, which is not fully carried out in the literature, is to incorporate distributional effects into the Calabi-Yau monopole equation:
\[
d(d\vartheta_0)= 2\pi\sum_i k_i L_i
\]
where $L_i$ represent the currents defined by disjoint 3-dimensional  submanifolds, and $k_i\in \Z$. Compatibility with the Calabi-Yau monopole condition requires $L_i$ to be \textbf{special Lagrangians}, namely \[
\omega_{CY}|_{L_i}=0, \quad \text{Im}\Omega_{CY}|_{L_i}=0.
\]
The cohomological condition $\sum k_i[L_i]=0$ is known as \textbf{charge conservation}. Solving the Calabi-Yau monopole equation, $d*d\mathfrak{h}$ is equal to the signed measure 
$v\mapsto -2\pi\sum k_i\int_{L_i} v \text{Re}\Omega_{CY}$ supported on the 3-cycles, so $\mathfrak{h}\sim \frac{k_i}{2\text{dist}(L_i, \cdot)}$ near $L_i$. The curvature form $d\vartheta$ integrates to $2\pi k_i$ on the suitably oriented 2-spheres linking $L_i$.

As an important variant, we allow for Calabi-Yau orbifolds with $\Z_2$ quotient singularities locally modelled on the fixed point set of an antiholomorphic involution, so the orbifold singular sets $L_i$ are special Lagrangians. In such cases the modification is that $\mathfrak{h}\sim \frac{2k_i}{2\text{dist}(L_i, \cdot)  }$.

Geometrically, the asymptotes of $(\mathfrak{h}, \vartheta_0)$ near $L_i$ means that the $G_2$-structure $\phi_\epsilon$ transverse to $L_i$ matches approximately with the asymptotes for ALF gravitational instantons, of $A_{k_i-1}$-type for $k_i\geq 0$ in the manifold locus case, and respectively of $D_{k_i+2}$-type for $k_i\geq -2$ in the orbifold locus case. The strategy is then to desingularize the neighbourhood of $L_i$ by gluing in a suitable fibration of ALF instantons over $L_i$. The meaning of some special values of $k_i$ are as follows:
\begin{itemize}
\item The $A_{-1}$-type ALF instanton is the flat product $S^1\times\R^3$. In this case $k_i=0$, so the singularity does not actualy exist.

\item The $A_0$-type ALF instanton is the Taub-NUT metric, which has no deformation once we fix the asymptotic circle length. One expects the ALF fibration to be essentially uniquely determined by the asymptotic matching requirement. Similarly with the $D_{-2}$-type ALF instanton, also known as the Atiyah-Hitchin metric.

\item In other cases, the ALF instantons have nontrivial moduli, so one expects internal degrees of freedoms to arise in the ALF fibration, similar to the harmonic one-forms appearing in the work of Joyce and Karigiannis \cite{JoyceKarigiannis}.

\end{itemize}

Notice charge conservation necessitates the appearance of some $L_i$ with $k_i<0$. The corresponding dihedral type ALF gravitational instantons then break the $S^1$-symmetry of the $G_2$-structure, thus making the compact examples possible.




\subsection{Iterated Collapse I: fast circle collapsing}\label{fastcirclecollapsing}

We consider now the degenerate situation where the Calabi-Yau 3-fold $M/S^1$ itself admits a holomorphic K3 fibration over a Riemann surface $S$, whose fibres shrink down as we vary the Calabi-Yau structure. It is natural to expect that if the circle collapsing happens at a much faster rate than the shrinking of the K3 fibres, then the Foscolo-Haskins-Nordstr\"om picture of $S^1$-invariant collapsing $G_2$-metrics should still be valid. We shall make a formal analysis for the range of parameters
\[
\text{diam}(S) \sim 1, \quad \text{diam}(K3)\sim t \ll 1, \quad \text{diam}(S^1)\sim \epsilon \ll t^2,
\]
and over a local region of $S$ where all K3 fibres are smooth.

 The leading order behaviour of the Calabi-Yau metric on $M/S^1$ is 
 \[
 \omega\approx \omega_S+ t^2 \omega_y, \quad \Omega\approx t^2 dy\wedge \Omega_y,
 \]
 where $y$ is a local holomorphic coordinate of $S$, and $\omega_y, \Omega_y$ are a family of K\"ahler metrics and holomorphic 2-forms on the K3 fibres parametrized by $y$. We are also implicitly using the horizontal distribition defined by the orthogonal complements of the tangent space of the K3 fibres, which induces a horizontal-vertical type decomposition on differential forms over $M/S^1$. In the $t\to 0$ limit $\omega_y$ is the Calabi-Yau metric on the K3 fibres in the class $t^{-2}[\omega]\in H^2(K3)$, so $\omega_y, \Omega_y$ define a hyperk\"ahler structure. Without loss of generality $t^{-4}\int_{K3}[\omega]^2=1$.

Assuming that $d\vartheta$ has a smooth limit $d\vartheta_0$ as $t\to 0$, the condition $d\vartheta\wedge \omega^2=0$ becomes in the limit
\[
d\vartheta_0\wedge \omega_y=0 
\] 
on all K3 fibres, and in particular $[d\vartheta]\wedge [\omega]=0\in H^4(K3)$. We then analyze the condition $\frac{1}{2}dh\wedge \omega^2= \epsilon h^{1/4} d\vartheta\wedge \text{Im}\Omega$ by decomposition into horizontal-vertical types:
\begin{itemize}
\item 
The vertical derivative of $h$ along the K3 surfaces is of order $O(\epsilon)$. This allows us to treat $h$ as if it only depends on $y\in Y$. We can write $y=y_1+ \sqrt{-1}y_2$, and $\Omega_y= \omega_2+\sqrt{-1}\omega_1$, so $t^{-2}\text{Im}\Omega= \omega_1dy_1+ \omega_2dy_2$.

\item 
On the K3 surfaces, 
\[
\frac{1}{2} \frac{\partial h}{\partial y_i  } \omega_y^2\approx \epsilon t^{-2} h d\vartheta_0 \wedge \omega_i, \quad i=1,2.
\]
\end{itemize}
We see that the key assumption $h\approx 1$ in \cite{HaskinsS1} is only valid if $\epsilon\ll t^2$. When this holds, then

\[
dh\approx 2\epsilon t^{-2} \sum dy_i \int_{K3} [d\vartheta]\wedge \omega_i ,  
\]
and the restriction of $d\vartheta_0$ to the K3 fibres are the unique harmonic 2-form in the fixed class $[d\vartheta]$, because the self dual part is prescribed above.

We now seek  special Lagrangians in the adiabatic setting. For simplicity consider a 3-manifold $L$ inside $M/S^1$ fibred over a curve $l$ in $S$, whose fibres are 2-spheres $L_y$ representing a $(-2)$-class $\sigma\in H^2(K3)$. We require
\[
\omega_y |_{L_y}=0, \quad \text{Im}\Omega|_L=0,
\]
so along $l$,
\begin{equation}\label{specialLagrangianadiabatic}
[\omega]\cdot \sigma=0, \quad \sigma\cdot \sum_i[\omega_i]dy_i=0,
\end{equation}
meaning $\sigma$ is of type $(1,1)$ for a particular choice of complex structure on the K3 surfaces. Setting $L_y$ as the $(-2)$-curve in the class $\sigma$, we obtain $L=\cup_{y\in l} L_y$.

\subsection{Iterated Collapse II: fine tuned circle collapsing}\label{Slowcollapsing}

We now turned to the fine tuned scaling $\epsilon=t^2$, which is the main setting of this paper (the case of $\epsilon/t^2=\text{const}$ can be reduced to this by redefining $h$). The basic distance scales are
\[
\text{diam}(S) \sim 1, \quad \text{diam}(K3)\sim t \ll 1, \quad \text{diam}(S^1)\sim \epsilon= t^2.
\]
Our goal is to find an ansatz solving the Apostolov-Salamon equation (\ref{AScollapsing}) approximately, and encode it by adiabatic data. This procedure has strong analogy with Donaldson's proposal about coassociative K3 fibrations \cite{Donaldson}. In this section the base $S$ is local.

In analogy with the $\epsilon\ll t^2$ case, we wish to maintain the following features in the adiabatic limit:
\begin{itemize}
\item The 6-fold $M/S^1$ with $SU(3)$-structure $(\omega, \Omega)$ is fibred by K3 surfaces over a surface $S$, and the tangent spaces of the K3 fibres are preserved by the almost complex structure. We write $y_1, y_2$ as the local real coordinates on $S$, and
\[
\begin{cases}
t^{-2}\text{Im}\Omega=\omega_1dy_1+\omega_2dy_2,
\\
\omega= t^2\omega_y+ \omega_S,
\end{cases}
\]
using the horizontal distribution induced by the $SU(3)$-structure. We impose the cohomological normalisation $\int_{K3}[\omega]^2=t^4$, namely $[\omega_y]^2=1$.

\item
The positive valued function $h$ depends only on $y_1, y_2$ to leading order, and the vertical derivative along K3 fibres are negligible.

\end{itemize}

We will derive the ansatz by formal calculations using (\ref{AScollapsing}):
\begin{itemize}
\item 
By $d\text{Im}\Omega= -\frac{1}{4}h^{-1}dh\wedge \text{Im}\Omega$, the 2-forms $\omega_1, \omega_2$ restricted to the K3 fibres are closed, and 
\[
\frac{\partial }{\partial y_2}(h^{1/4}[\omega_1] )= \frac{\partial }{\partial y_1}(h^{1/4}[\omega_2] ).
\]
Over the local base, we find a function $H: S\to H^2(K3)$, with
\begin{equation}
\frac{\partial H}{\partial y_i}=h^{1/4}[\omega_i], \quad i=1,2.
\end{equation}
The $SU(3)$-condition implies $\omega_y \wedge \omega_i=0$ on K3 surfaces, so $\frac{\partial H}{\partial y_i}\wedge [\omega]=0\in H^4(K3)$, and up to choosing an additive constant $H\wedge [\omega]=0$.

\item 
By $d(h^{3/4}\text{Re}\Omega)=-\epsilon d\vartheta \wedge \omega$, writing $t^{-2}\text{Re}\Omega= \Theta_1 dy_1+ \Theta_2 dy_2$, we deduce to leading order on the K3 fibres,
\[
d\Theta_1=0, \quad d\Theta_2=0.
\]
From the SU(3)-structure $\omega_y \wedge \Theta_i=0$, and $\Theta_i$ are closed self dual 2-forms on the K3 fibres. Since $\omega_1, \omega_2$ are also closed self dual 2-forms orthogonal to $\omega_y$, it follows that $\Theta_1, \Theta_2$ are $\R$-linear combinations of $\omega_1, \omega_2$. The linear coefficients can be determined by a cohomological calculation. Denote
\[
g_{ij}= \int_{K3} \omega_i\wedge \omega_j, \quad i,j=1,2, \quad \det g=\det (g_{ij}),
\]
then the $SU(3)$-structure requirement implies 
\begin{equation}\label{Thetaintermsofomega}
t^{-2}\text{Re}\Omega= \sqrt{\det g} \{   dy_1 (g^{12}\omega_1+ g^{22}\omega_2)- dy_2 (g^{11}\omega_1+ g^{21}\omega_2)     \}.
\end{equation}
As a simple check, if $y_1, y_2$ are isothermal coordinates so that $g_{ij}\propto \delta_{ij}$, then
\begin{equation*}
t^{-2}\text{Re}\Omega=   dy_1 \omega_2- dy_2 \omega_1, \quad t^{-2}\Omega= (\omega_2+ \sqrt{-1}\omega_1 )(dy_1+ \sqrt{-1}dy_2)    
\end{equation*}
as expected. In particular, the K3 fibres are endowed with a \textbf{hyperk\"ahler structure} via $(\omega_y,  \omega_2+ \sqrt{-1}\omega_1)$.  Since $\omega^3= \frac{3}{2}\text{Re}\Omega\wedge \text{Im}\Omega$, to leading order 
\[
3t^4 \omega_S \wedge \omega_y^2= \frac{3}{2}  \text{Re}\Omega\wedge \text{Im}\Omega= 3t^4\sqrt{\det g} dy_1\wedge dy_2\wedge \omega_y^2, 
\]
hence $\omega_S\approx \sqrt{\det g} dy_1\wedge dy_2$.

\item
By $\frac{1}{2}dh\wedge \omega^2= \epsilon h^{1/4} d\vartheta \wedge \text{Im}\Omega$, we see on the K3 fibres
\[
\frac{\partial h}{\partial y_i} \omega_y^2= 2h^{1/4} d\vartheta\wedge \omega_i, \quad i=1,2.
\]
Taking the cohomology classes,
\[
\frac{\partial h}{\partial y_i}= 2h^{1/4}\int_{K3}[d\vartheta] \wedge [\omega_i]= 2 \int_{K3} \frac{\partial H}{\partial y_i} \wedge [d\vartheta],
\]
so by a choice of additive constant for $H$,
\begin{equation}\label{hviaH}
h=\begin{cases}
2[d\vartheta] \cdot H, \quad & [d\vartheta]\neq 0,
\\
\text{const}>0, \quad & [d\vartheta]=0.
\end{cases} 
\end{equation}
This additive choice is compatible with $H\cdot [\omega]=0$ because $[d\vartheta]\cdot [\omega]=0$. We require $h>0$.

\item 
By $d\vartheta\wedge \omega^2=0$, to leading order $d\vartheta\wedge \omega_y=0$ on the K3 fibres. We see that the self dual part of $d\vartheta$ on the K3 fibres are completely prescribed, so $d\vartheta$ is the unique \textbf{harmonic 2-form} in its class with respect to the hyperk\"ahler structure on the K3 fibre.

\item
The function $H$ maps to the polarized subspace $[\omega]^\perp\subset H^2(K3)$, which is of signature $(2, 19)$. Now the metric induced on $S$ by the immersion $H$ is
\begin{equation}
\tilde{g}= \frac{\partial H}{\partial y_i} \cdot \frac{\partial H}{\partial y_j} dy_idy_j=   h^{1/2} g_{ij}dy_idy_j,
\end{equation}
which is positive definite. Thus we call $H$ a \textbf{polarized positive section}, in close analogy with Donaldson's proposal \cite{Donaldson}.

\item
Since in this local discussion we assume there is no singular fibre, cohomologically we require there is no $(-2)$-class $\sigma$ with
\[
\frac{\partial H}{\partial y_i}\wedge \sigma=0, \quad i=1,2, \quad [\omega_y]\wedge \sigma=0
\]
at any point on our local base. We say the polarized positive section $H$ \textbf{avoids excess $(-2)$-classes}.

\item
We take a closer look at $d(h^{3/4}\text{Re}\Omega)=-\epsilon d\vartheta \wedge \omega$. Our previous requirements imply $\epsilon^{-1}d(h^{3/4}\text{Re}\Omega)$ has horizontal-vertical type $(2, 2)$, so it defines an $H^2(K3)$-valued 2-form on $S$,
\[
\{ \frac{\partial}{\partial y_2}( h^{3/4}[\Theta_1])- \frac{\partial}{\partial y_1}( h^{3/4}[\Theta_2]) \} dy_2\wedge dy_1= \frac{\partial}{\partial y_j}(h^{1/2} g^{ij}\sqrt{\det g} \frac{\partial H}{\partial y_i}  )dy_2\wedge dy_1.
\]
Now $d\vartheta\wedge \omega$ is to leading order $d\theta\wedge \omega_S$, so defines the $H^2(K3)$-valued 2-form
\[
[d\vartheta] \omega_S= [d\vartheta] \sqrt{\det g} dy_1\wedge dy_2.
\]
Comparing the two expressions,
\begin{equation}\label{weightedmaximalsubmanifold}
\frac{\partial}{\partial y_j}(h^{1/2} g^{ij}\sqrt{\det g} \frac{\partial H}{\partial y_i}  )= [d\vartheta] \sqrt{\det g} .
\end{equation}
This is a PDE on the polarized positive section $H$, and we call it the \textbf{weighted maximal submanifold equation} in view of its close analogy with the maximal submanifold equation in \cite{Donaldson}. In fact, when $[d\vartheta]=0$, then $h$ is a positive constant, and this is the 2-dimensional version of the usual maximal submanifold equation, namely the image of $H$ has mean curvature zero.
\end{itemize}

We now discuss the \textbf{reconstruction problem}: given a local polarized positive section $H: S\to H^2(K3)$ avoiding excess $(-2)$-classes, so that
\[
H\cdot [\omega_0]=0, \quad h=2H\cdot [d\vartheta]>0, \quad  \tilde{g}= \frac{\partial H}{\partial y_i} \cdot \frac{\partial H}{\partial y_j} dy_idy_j \text{ Riemannian},
\]
for some $[\omega_0]\in H^2(K3)$ with $[\omega_0]^2>0$, and given an integral class $\frac{1}{2\pi}[d\vartheta]\in H^2(K3)$ with $[\omega_0]\cdot [d\vartheta]=0$, how can we reverse engineer the data $(\omega, \Omega, h, \vartheta)$ in the adiabatic limit? (If $[d\vartheta]=0$, then we should modify to $h=\text{const}>0$ instead, and $H$ is just a positive section constrained to be orthogonal to $[\omega_0]$.) We will describe such a procedure that is nearly canonical.

\begin{itemize}
\item  Since we are working over a local base $S$, topologically the K3 fibration is just the trivial fibration. The $S^1$-bundle is topologically determined by its Chern class $\frac{1}{2\pi}[d\vartheta]$.

\item 

Suppose we are given the polarized positive section $H: S\to H^2(K3)$. As a first step, we use the global Torelli theorem on K3 surfaces to construct the hyperk\"ahler structure on the K3 surfaces, such that
\[
[\omega_y]= [\omega_0], \quad [\omega_1]= h^{-1/4} \frac{\partial H}{\partial y_1}, \quad [\omega_2]= h^{-1/4} \frac{\partial H}{\partial y_2}.
\]
This makes sense because the 3-subspace spanned by $[\omega_y], [\omega_1], [\omega_2]$ is a positive 3-subspace inside $H^2(K3)$, by the polarized positive section condition. The resulting hyperk\"ahler structures vary smoothly with the K3 fibres. 
We have $[\omega_i]\cdot [\omega_j]= g_{ij}= h^{-1/2}\tilde{g}_{ij}$. This allows us to construct $\Theta_1, \Theta_2$ as linear combinations of $\omega_1, \omega_2$ via (\ref{Thetaintermsofomega}).

\item 
Now to regard $\omega_y, \omega_1, \omega_2$ as forms on the 6-fold $M/S^1$ requires us to specify a \textbf{horizontal distribution}. The difference between two horizontal distributions is given by a 1-form on $S$ with values in the space of vector fields on the fibre K3 surfaces, which modifies $\omega_y, \omega_i$ (modulo terms which involve $dy_1\wedge dy_2$) by
\[
\omega_y\to \omega_y+ \iota_{v_1} \omega_y dy_1+ \iota_{v_2} \omega_y dy_2, \quad \omega_i\to \omega_i+ \iota_{v_1} \omega_i dy_1+ \iota_{v_2} \omega_i dy_2,\quad  i=1,2.
\]
We start with an arbitrary horizontal distribution, to get a preliminary choice of $\omega_y, \text{Im}\Omega, \text{Re}\Omega$. Using the above modification, we can make $d\omega_y=0$ modulo terms involving $dy_1\wedge dy_2$. This reduces the ambiguity of the horizontal distribution to 
\[
\iota_{v_i}\omega_y= da_i, \quad i=1,2,
\]
for real valued functions $a_1, a_2$ on the K3 surfaces. We claim we can spend up this freedom to make $d(h^{1/4}\text{Im}\Omega)=0$.

Notice first that  by construction $d(h^{1/4}\text{Im}\Omega)$ is a 2-form on $S$ valued in the space of exact 2-forms on the K3 fibres; there is no horizontal-vertical type $(1, 3)$-piece because $\omega_1,\omega_2$ are closed on the K3, and the 2-form on the K3 fibres are exact instead of closed because 
\[
\frac{\partial }{\partial y_2} (h^{1/4}[\omega_1])= \frac{\partial }{\partial y_1}(h^{1/4}[\omega_2]).
\]
Morever, since $\omega_y\wedge \text{Im}\Omega=0$, we get
\[
\omega_y\wedge d(h^{1/4}\text{Im}\Omega)= d( \omega_y\wedge h^{1/4}\text{Im}\Omega)=0,
\] 
so the exact 2-forms are fibrewise orthogonal to $\omega_y$. It then suffices to show that by adjusting the horizontal distribution, we can impose the extra conditions
\[
d(h^{1/4}\text{Im}\Omega)\wedge \omega_i= 0, \quad i=1,2.
\]
This would force $d(h^{1/4}\text{Im}\Omega)=0$ because an exact anti-self-dual 2-form is zero.

To see this new claim, notice under the modification of the horizontal distribution with $\iota_{v_i}\omega_y=da_i$, 
\[
\text{Im}\Omega\to \text{Im}\Omega+ (\iota_{v_1}\omega_2- \iota_{v_2}\omega_1) dy_1 \wedge dy_2. 
\]
We need to choose $a_1, a_2$ so that the exact 4-forms on the K3 fibres appearing in $d(h^{1/4}\text{Im}\Omega)\wedge \omega_i$ are cancelled by
\[
d( \iota_{v_1}\omega_2- \iota_{v_2}\omega_1  ) \wedge \omega_i, \quad i=1,2.
\]
This problem is easiest to analyze in isothermal coordinates using the hyperk\"ahler structure on the K3 surfaces, where up to constants
\[
d( \iota_{v_1}\omega_2- \iota_{v_2}\omega_1  ) \wedge \omega_i \propto d*da_i,
\]
which can indeed realize any pair of exact 4-forms by a suitable choice of $a_1, a_2$. We have precisely fixed all the ambiguity of the horizontal distribution.

\item
We can write $t^{-2}d(h^{3/4}\text{Re}\Omega)= dy_2\wedge dy_1\wedge \Theta_3$ since the horizontal-vertical type (1,3) component vanishes. The 2-form $\Theta_3$ is closed on each K3 fibre, and lies in the cohomology class $\frac{\partial}{\partial y_j}(h^{1/2} g^{ij}\sqrt{\det g} \frac{\partial H}{\partial y_i}  )$. We claim that on each fibre it is the unique \textbf{harmonic 2-form} in this class.

We start with the fact that on a 6-manifold, if $\text{Re}\Omega$ and $\text{Im}\Omega$ are related as the real and imaginary parts of a complex 3-form, then for any vector field $w$,
\[
\iota_w d\text{Im}\Omega \wedge \text{Re}\Omega= \iota_w \text{Im}\Omega \wedge d\text{Re}\Omega.
\]
Applying this to $h^{1/4}\text{Re}\Omega$ and $h^{1/4}\text{Im}\Omega$, and taking $w$ as the horizontal lifts of $\frac{\partial}{\partial y_i}$, we see from $d( h^{1/4}\text{Im}\Omega )=0$ that
$
d(h^{1/4}\text{Re}\Omega) \wedge \omega_i=0.
$
Thus 
\[
d( h^{3/4} \text{Re}\Omega  )\wedge \omega_i= \frac{1}{2}h^{-1/4}dh \wedge (\text{Re}\Omega\wedge \omega_i)= \frac{t^2}{2}h^{-1/4}\frac{\partial h}{\partial y_i} \sqrt{\det g}dy_2\wedge dy_1\wedge \omega_y^2,
\]
or equivalently $\Theta_3\wedge \omega_i= \frac{1}{2}h^{-1/4}\frac{\partial h}{\partial y_i} \sqrt{\det g} \omega_y^2$ on each K3 fibre.

Another relation comes from $\text{Re}\Omega\wedge \omega_y=0$, and gives upon differentiation that $\Theta_3\wedge \omega_y=0$ on each K3 fibre. These relations completely prescribe the self dual part of $\Theta_3$, implying it is the harmonic form.

\item
Notice $d\omega_y$ is a 2-form on $S$ valued in the space of closed 1-forms on the K3 fibres. Since $H^1(K3)=0$, the 1-forms are exact, so up to adjusting $\omega_y$ by a function times $dy_1\wedge dy_2$, we can make $\omega_y$ closed on $M/S^1$. Unlike all previous constructions, this step is not fully canonically determined by the topological data and the polarized positive section $H$. The ambiguity in $\omega_y$ is a 2-form pulled back from $S$.

\item
Let $\omega_S= \sqrt{\det g} dy_1\wedge dy_2$, and set $\omega=t^2\omega_y+ \omega_S$. Then $d\omega=0$ on $M/S^1$, and $(\omega, \Omega)$ satisfy all other requirements of the $SU(3)$-structure precisely, except for the normalisation condition which is only approximate:
\[
\omega^3= \frac{3}{2} (1+O(\epsilon)) \text{Re}\Omega\wedge \text{Im}\Omega.
\]
Notice the construction of $\Omega$ is completely canonical, and $\omega$ is canonical up to $\epsilon$ times a 2-form pulled back from $S$; we use this freedom to make the $O(\epsilon)$ term have zero fibrewise integral.

\item
The restrictions of $d\vartheta$ to the K3 fibres are prescribed to be the harmonic 2-forms in the class $[d\vartheta]$. By construction, when restricted to the K3 fibres
\[
d\vartheta\wedge \omega_y=0, \quad 2h^{1/4}d\vartheta\wedge \omega_i= \frac{\partial h}{\partial y_i}\omega_y^2.
\]
This determines $\vartheta$ up to $f_1dy_1+f_2dy_2$ for some arbitrary smooth functions $f_1, f_2$, and gauge equivalence.

\item

Notice the quantity
\[
\frac{1}{2} dh\wedge \omega^2-\epsilon h^{1/4} d\vartheta \wedge \text{Im}\Omega=  \frac{\epsilon^2}{2} dh\wedge \omega_y^2-\epsilon^2 h^{1/4} d\vartheta \wedge (\omega_1dy_1+ \omega_2dy_2)  
\]
can be written as $\epsilon^2 dy_1\wedge dy_2\wedge \Theta_4$ where $\Theta_4$ defines fibrewise 3-forms on K3 surfaces, because by construction the horizontal-vertical (1,4)-type component vanishes. Since 
\[
d\omega=0, \quad d(h^{1/4} \text{Im}\Omega)=0,
\]
we see $d\Theta_4=0$ on the K3 fibres, and by $H^3(K3)=0$ the $\Theta_4$ is exact. Any such $\Theta_4$ on a hyperk\"ahler K3 can be written as
\[
\Theta_4= df_1\wedge \omega_1+ df_2\wedge \omega_2+ df_3\wedge \omega_y.
\]
Spending up the freedom in $d\vartheta$, we can cancel out the $f_1, f_2$ terms, and demand $\Theta_4= df_3\wedge \omega_y$ for some function $f_3$.

Thus
\[
\frac{1}{2} dh\wedge \omega^2-\epsilon h^{1/4} d\vartheta \wedge \text{Im}\Omega= \epsilon^2 dy_1\wedge dy_2\wedge \omega_y\wedge df_3,
\]
which is suppressed by the factor $\epsilon^2$. With respect to the metric defined by the $G_2$-structure $\phi$, this quantity has small magnitude of order $O(t)$. This measures the error of the coclosed condition of the $G_2$-structure.

\item
Now suppose the polarized positive section $H$ satisfies the weighted maximal submanifold equation (\ref{weightedmaximalsubmanifold}). By construction,
\[
d(h^{3/4}\text{Re}\Omega)=\epsilon \Theta_3 \wedge dy_2\wedge dy_1= \epsilon d\vartheta \wedge \sqrt{\det g} dy_2\wedge dy_1= -\epsilon d\vartheta \wedge \omega_S.
\]
Then the quantity 
\[
d(h^{3/4}\text{Re}\Omega)+ \epsilon d\vartheta\wedge \omega= \epsilon^2 d\vartheta\wedge \omega_y
\]
is of type $(1,3)+(2,2)$, and its magnitude with respect to the metric defined by the $G_2$-structure $\phi$ is small of order $O(t)$. This measures the error of the closed condition of the $G_2$-structure. We emphasize that only in this last step do we bring in the weighted maximal submanifold equation.



\end{itemize}

In conclusion, we produced an $S^1$-invariant ansatz $G_2$-structure $\phi$ from the data of a polarized positive section satisfying the weighted maximal submanifold equation, which is approximately closed and coclosed.

\begin{rmk}
The adiabatic special Lagrangian condition (\ref{specialLagrangianadiabatic}) reads in this $\epsilon=t^2$ context
\[
[\omega]\cdot \sigma=0, \quad \sigma\cdot H=\text{const}.
\]
It makes sense even though $M/S^1$ loses the interpretation as a Calabi-Yau 3-fold.
\end{rmk}

\begin{rmk}
The geometrical significance of our ansatz is that it is a plausible approximation for collapsing $G_2$-manifolds in the generic region. 
When we deform the cohomology class of the $G_2$-structure, so that the circle collapsing happens at an even slower rate $\epsilon\gg t^2$ compared to the K3 collapsing, then plausibly the $S^1$-symmetry reduction hypothesis will be invalidated, and higher Fourier modes in the $S^1$-direction start to be significant. 
\end{rmk}

\subsection{Variational formulation}\label{Variationalweightedmaximal}

We now explain how to see the weighted maximal submanifold equation as an Euler Lagrange equation. Given a polarized positive section $H: S\to H^2(K3)$, we define the \textbf{weighted area functional} as
\[
\mathcal{A}_w(H)=  \int_S h^{1/2} \sqrt{\det \tilde{g} }dy_1\wedge dy_2, \quad h= 2[d\vartheta]\cdot H, \quad \tilde{g}_{ij}= \frac{\partial H}{\partial y_i}\cdot \frac{\partial H}{\partial y_j}.
\]

We calculate the first variation. Given a variation $\delta H=f$, then
\[
\delta h= 2[d\vartheta]\cdot f, \quad \delta \tilde{g}_{ij}= \frac{\partial H}{\partial y_i}\cdot \frac{\partial f}{\partial y_j}+  \frac{\partial H}{\partial y_j}\cdot \frac{\partial f}{\partial y_i},
\]
\[
\delta \sqrt{\det (\tilde{g})}= \frac{1}{2}\sqrt{\det (\tilde{g})} \tilde{g}^{ij} \delta \tilde{g}_{ij}= \sqrt{\det (\tilde{g})} \tilde{g}^{ij} \frac{\partial H}{\partial y_i}\cdot \frac{\partial f}{\partial y_j},
\]
Hence the \textbf{first variation formula}
\begin{equation}
\begin{split}
\delta \mathcal{A}_w= &\int h^{1/2} \delta \sqrt{\det (\tilde{g})} dy_1dy_2+ \int \frac{1}{2} h^{-1/2}\delta h \sqrt{\det (\tilde{g})} dy_1dy_2
\\
=& \int h^{1/2}\sqrt{\det g} {g}^{ij} \frac{\partial H}{\partial y_i}\cdot \frac{\partial f}{\partial y_j}  dy_1dy_2+ 
\int  [d\vartheta]\cdot f \sqrt{\det g} dy_1dy_2
\\
=&  \int \{ [d\vartheta]\sqrt{\det g}- \frac{\partial }{\partial y_j}  (h^{1/2}\sqrt{\det g} {g}^{ij} \frac{\partial H}{\partial y_i})\}\cdot f dy_1dy_2.
\end{split}
\end{equation}
The critical points exactly reproduce the weighted maximal submanifold equation.

\begin{rmk}\label{areafunctional}
The terminology `weighted area' is of course based on the area functional
\[
\mathcal{A}(H)=  \int_S  \sqrt{\det \tilde{g} }dy_1\wedge dy_2, \quad  \tilde{g}_{ij}= \frac{\partial H}{\partial y_i}\cdot \frac{\partial H}{\partial y_j},
\]
whose critical points are called maximal submanifolds, because these are local maxima of the area functional due to the signature of the setup \cite{Donaldson}\cite{Limaximal}. Such objects (over a 3-manifold) are central to Donaldson's proposal to adiabatic coassociative K3 fibrations \cite{Donaldson}. In fact, when $[d\vartheta]=0$, then in the discussion above we should use $h=\text{const}>0$, and then $\mathcal{A}_w$ is just up to a multiplicative constant the area functional.
\end{rmk}

\begin{rmk}
In the ansatz of section \ref{Slowcollapsing}, the volume form is approximately
\[
dvol_\phi\sim \frac{1}{2} \epsilon^3 h^{1/2}\sqrt{\det \tilde{g} } \vartheta \wedge dy_1 \wedge dy_2 \wedge  \omega_y^2,
\]
so the volume is approximately
$
\text{Vol}\sim  \pi\epsilon^3 \mathcal{A}_w(H).
$
The variational formulation can be then expected from Hitchin's characterisation of the torsion free $G_2$-structures in terms of the critical points of the volume functional.
\end{rmk}

We now look at the second variation $\mathcal{L}$ of the weighted area functional at some critical point $H$, namely
\[
\mathcal{A}_w(H+s f)= \mathcal{A}_w(H) + \frac{s^2}{2}\int_S f\mathcal{L}f+O(s^3),\quad s\to 0. 
\] 
Here $\mathcal{L}f$ takes value in $H^2(K3)$-valued 2-forms.
This is naturally expected to be important in the deformation theory of our collapsing $G_2$-manifolds. Since $\mathcal{A}_w$ is diffeomorphism invariant on $S$, we know that only the normal component of $f$ is relevant for the second variation, and $\mathcal{L}f$ is in particular normal.

\begin{lem}\label{secondvariationnegative}
Let $S$ be a 2-manifold with boundary. Then $\int_S f\mathcal{L}f$ is negative definite as a symmetric form on normal vector fields $f$ vanishing on the boundary.	
\end{lem}

\begin{proof}
Consider the images of a family of maps $H+sN$ for $N$ some normal vector field to the image of $H$, inducing the metrics $\tilde{g}(s)$ from the immersion into $H^2(K3)$. Standard calculation from submanifold theory gives 
\[
\frac{\partial}{\partial s}|_{s=0} \sqrt{ \det \tilde{g} } dy_1\wedge dy_2= - \langle \mathfrak{m}, N\rangle \sqrt{ \det \tilde{g} } dy_1\wedge dy_2 
\]
where $\mathfrak{m}=(\sum_i\nabla_{e_i} e_i)^{\perp}$ is the mean curvature, and $e_i$ stands for an orthonormal frame, and $\perp$ means the normal projection. The second derivative is
\[
\begin{split}
&\frac{\partial^2}{\partial s^2}|_{s=0} \sqrt{ \det \tilde{g}(s) } dy_1\wedge dy_2
\\
&= (- \Tr (II_N^2) + \sum_i \langle (\nabla_{e_i} N)^\perp, (\nabla_{e_i} N)^\perp \rangle +\langle \mathfrak{m}, N\rangle^2 )\sqrt{ \det \tilde{g} } dy_1\wedge dy_2
\end{split}
\]
where $II_T(e_i, e_j)= \langle N, \nabla_{e_i} e_j\rangle$ is the $N$-component of the second fundamental form. Morever,
\[
\frac{\partial }{\partial s}|_{s=0} h(s)^{1/2}=  h^{-1/2} \langle N, [d\vartheta] \rangle,
\]
\[
\frac{\partial^2 }{\partial s^2}|_{s=0} h(s)^{1/2}= -  h^{-3/2} \langle N, [d\vartheta] \rangle^2.
\]
Since the first variation of $h^{1/2} \sqrt{\det \tilde{g}}$ is zero at $H$, we have
\[
- h^{1/2}\langle \mathfrak{m}, N\rangle + h^{-1/2} \langle N, [d\vartheta]\rangle=0.
\]
Thus by the Leibniz rule
\[
\begin{split}
& \frac{\partial^2 }{\partial s^2}|_{s=0}
(h(s)^{1/2}\sqrt{ \det \tilde{g}(s) } dy_1\wedge dy_2 )
\\
=&
h^{1/2} (- \Tr (II_N^2) + \sum_i \langle (\nabla_{e_i} N)^\perp, (\nabla_{e_i} N)^\perp \rangle -2\langle \mathfrak{m}, N\rangle^2 )\sqrt{ \det \tilde{g} } dy_1\wedge dy_2.
\end{split}
\]
The \textbf{second variation formula} is then
\begin{equation}
\int_S N\mathcal{L}N= \int_S (- \Tr (II_N^2) + \sum_i \langle (\nabla_{e_i} N)^\perp, (\nabla_{e_i} N)^\perp \rangle -2\langle \mathfrak{m}, N\rangle^2 ) h^{1/2}\sqrt{ \det \tilde{g} } dy_1\wedge dy_2.
\end{equation}
Crucially, in the normal direction, the ambient metric is negative definite, so all the 3 terms are negative. The equality requires $II_N=0$ and $(\nabla_{e_i} N)^\perp=0$, so $\nabla_{e_i}N=0$, and $N$ is parallel along $S$. The zero boundary condition then forces $N$ to be zero.
\end{proof}

In particular, the boundary value problem of the weighted maximal submanifold equation will have unobstructed deformation under the change of boundary data.

\subsection{Instrinsic geometry of the weighted maximal submanifold}

As $\epsilon\to 0$, the family of (approximate) $G_2$-metrics collapse down to the base $S$ in the Gromov-Hausdorff sense, whose natural structure is a \textbf{metric measure space}.
Since $g_\phi\approx h^{1/2} g_{M/S^1}+\epsilon^2 h^{-1}\vartheta^2$,
the limiting metric is
\[
\tilde{g}= h^{1/2}g= \frac{\partial H}{\partial y_i} \cdot \frac{\partial H}{\partial y_j} dy_i\otimes dy_j,
\]
namely the induced metric from the immersion $H: S\to H^2(K3)$. The natural measure is the one corresponding to the weighted area functional, namely $h^{1/2} \sqrt{\det \tilde{g}}dy_1\wedge dy_2$. Notice this is different from the metric area form by a factor $h^{1/2}$. This is natural from the perspective of low energy effective action: when $\epsilon\to 0$, the low energy modes of
\[
\frac{1}{2}\int |du|_{g_\phi}^2 dvol_{g_\phi}
\]
correspond to $u$ being the pullback of a function on $S$, in which case the action is approximately
\[
\frac{\pi \epsilon^3}{2} \int_S |du|_{\tilde{g}}^2 h^{1/2}  \sqrt{\det \tilde{g}}dy_1\wedge dy_2,
\]
proportional to
\[
\frac{1}{2} \int_S |du|_{\tilde{g}}^2 h^{1/2}  \sqrt{\det \tilde{g}}dy_1\wedge dy_2.
\]

For a general Riemannian metric $g$ with a measure form $e^{-f}dvol_g$, there is a notion called the \textbf{Bakry-\'Emery Ricci curvature} \cite{BakryEmery}:
\[
Ric_f= Ric_g+ Hess(f).
\]
We briefly explain the motivations. First, the $f$-Laplacian is defined as
\[
\Lap_f= \Lap- \nabla f\cdot \nabla,
\]
which satisfies the integration by part formula
\[
\int \langle \nabla u, \nabla v\rangle e^{-f}dvol_g= \int u \Lap_f v e^{-f} dvol_g.
\]
The Bakry-\'Emery Ricci curvature naturally appears in the Bochner formula
\[
\frac{1}{2}\Lap_f |\nabla u|^2= |Hess(u)|^2+ \langle \nabla u, \nabla \Lap_f u\rangle + Ric_f(\nabla u, \nabla u).
\]

\begin{lem}
On the metric measure space $S$ with $(\tilde{g}, h^{1/2}\sqrt{\det \tilde{g}} dy_1dy_2)$, the Bakry-\'Emery Ricci curvature is positive semi-definite.	
\end{lem}

\begin{proof}
The usual Ricci curvature of $\tilde{g}$ can be expressed by the second fundamental form via the Gauss equation:
\[
Ric_{\tilde{g}}(e_i, e_j)=  \langle II(e_i, e_j), \mathfrak{m}\rangle- \sum_k\langle II(e_k, e_j), II(e_i, e_k)\rangle ,
\]
where $\mathfrak{m}=\sum (\nabla_{e_k }e_k)^{\perp}$, and $II(e_i,e_j)=(\nabla_{e_i}e_j)^\perp$. From the weighted maximal submanifold equation, by the proof of Lemma \ref{secondvariationnegative},
\[
\mathfrak{m}= h^{-1} [d\vartheta]^\perp.
\]
Next we compute the Hessian of $\log h$.
We have
\[
\nabla_{e_i} h= 2\langle [d\vartheta], e_i\rangle,
\]
so $\nabla h= 2([d\vartheta]-[d\vartheta]^\perp)$. The Hessian matrix is
\[
\langle\nabla_{e_i}\nabla h, e_j\rangle= 2\langle II(e_i,e_j), [d\vartheta]^\perp\rangle.
\]
Then $\nabla\log h= h^{-1} \nabla h$, and
\[
\begin{split}
&\langle\nabla_{e_i}\nabla \log h, e_j\rangle= -h^{-2}\nabla_{e_i}h\cdot \nabla_{e_j}h +h^{-1} \langle \nabla_{e_i}\nabla h, e_j\rangle
\\
=& -4h^{-2} \langle [d\vartheta],e_i\rangle \langle [d\vartheta],e_j\rangle +2h^{-1}\langle II(e_i,e_j), [d\vartheta]^\perp\rangle
\\
=& -4h^{-2} \langle [d\vartheta],e_i\rangle \langle [d\vartheta],e_j\rangle +2\langle II(e_i,e_j), \mathfrak{m}\rangle.
\end{split}
\]
The Bakry-\'Emery Ricci curvature is $Ric_{BE}=Ric_{\tilde{g}}- \frac{1}{2}\nabla^2 \log h$. Combining the above, the mean curvature term cancels, and
\[
Ric_{BE} (e_i, e_j)=  - \sum_k\langle II(e_k, e_j), II(e_i, e_k)\rangle + 2h^{-2} \langle [d\vartheta],e_i\rangle \langle [d\vartheta],e_j\rangle,
\]
or equivalently for any tangent vector $V$,
\begin{equation}
Ric_{BE} (V, V)=  -\sum_k \langle II(e_k, V), II(e_k, V)\rangle + 2h^{-2} \langle [d\vartheta],V\rangle ^2 .
\end{equation}
Crucially the ambient metric in the normal direction is negative definite, so both terms are nonnegative, as required.
\end{proof}

\begin{rmk}
When $[d\vartheta]=0$, so $h=\text{const}>0$, the above reduces to the non-negativity of the usual Ricci curvature on maximal submanifolds. This has been observed in \cite{Donaldson}.
\end{rmk}

\subsection{Dimensional reduction to Calabi-Yau case}\label{DimreductiontoCY}

There is a different way the Apostolov-Salamon framework relates to Calabi-Yau 3-folds: we can start with a (non-compact) Calabi-Yau 3-fold $X$ with $S^1$-symmetry, and take the product with $\R$. Notice here the Calabi-Yau does not appear as $M/S^1$, and such examples are nontrivial in the sense that the $S^1$-bundle is not flat.

We first briefly review the basic setup of Calabi-Yau metrics with holomorphic and Hamiltonian $S^1$-actions \cite[chapter 2]{SunZhang}. We have   the $S^1$-connection $\vartheta$ and the moment coordinate $\mu$, which parametrizes a family of K\"ahler quotients $D=X//S^1$, carrying a family of induced K\"ahler forms $\tilde{\omega}(\mu)$ and a fixed holomorphic volume form $\Omega_D$ on $D$. In our later case of interest $X$ will be a Calabi-Yau 3-fold, and $D$ will be a K3 surface. We can write the K\"ahler structure on $X$ as
\[
\begin{cases}
\omega'= \vartheta\wedge d\mu+\tilde{\omega},
\\
\Omega'= (hd\mu- \sqrt{-1}\vartheta)\wedge \Omega_D.
\end{cases}
\]
Here $h^{-1}$ is the norm squared of the Killing vector field. (In our convention $\vartheta$ corresponds to $-\Theta$ in \cite{SunZhang}.) The K\"ahler condition in this formalism leads to
\[
d\vartheta= -\partial_\mu \tilde{\omega}- d_D^c h \wedge d\mu,
\]
where $d_D$ is the exterior derivative along $D$, and $d_D^cf=Jd_Df$ on $D$. This requires an integrability condition
\[
\partial^2_\mu \tilde{\omega}+ d_Dd_D^c h=0 \quad \text{on $D$}.
\]
In fact, when the $S^1$-action is allowed to have fixed points, then one should also incorporate distributional terms supported on the discriminant locus, to be discussed later.
 Finally, the Calabi-Yau condition
reads
\[
\frac{\tilde{\omega}^{n-1} }{(n-1)!}= \frac{\sqrt{-1}^{(n-1)^2}}{2^{n-1}} h\Omega_D\wedge \overline{\Omega}_D.
\]
In the 3-fold case this specializes to
\[
\tilde{\omega}^2= \frac{1}{2} h\Omega_D\wedge \overline{\Omega}_D.
\]
After eliminating variables, we arrive at a key equation in \cite{SunZhang}
\[
\partial^2_\mu \tilde{\omega}+ d_Dd_D^c \frac{ 2 \tilde{\omega}^2}{ \Omega_D\wedge \overline{\Omega}_D  }=0 \quad \text{on $D$}.
\]

\begin{eg}\label{gravitaionalinstantons}
The prototype example is the case of complex dimension two, where this discussion specializes to the standard Gibbons-Hawking ansatz:
\[
\begin{cases}
\omega'= \vartheta\wedge d\mu+ h \frac{\sqrt{-1}}{2} d\zeta\wedge d\bar{\zeta},
\\
\Omega'= (hd\mu-\sqrt{-1}\vartheta)\wedge d\zeta,
\\
d\vartheta= - \partial_\mu h   \frac{\sqrt{-1}}{2} d\zeta\wedge d\bar{\zeta}- d_D^c h\wedge d\mu,
\end{cases}
\]
where $\zeta$ is an $S^1$-invariant holomorphic coordinate, and $h$ satisfies the Laplace equation except on some isolated points:
\[
(\partial_\mu^2+ 4\partial_\zeta \partial_{\bar{\zeta}} )h=0.
\]
For instance,
\[
h= A+ \frac{1}{2r}, \quad r^2= |\mu|^2+|\zeta|^2, \quad A=\text{const}>0
\]
corresponds to the Taub-NUT metric. The formula
\[
h= A+ \frac{k+1}{2r}, \quad r^2= |\mu|^2+|\zeta|^2, \quad A=\text{const}>0, \quad k=-1,0,1,\ldots
\]
describes the asymptote at infinity of $A_k$ type ALF gravitational instantons.

A small variant of the construction is to take the $\Z_2$-quotient under $(\mu,\zeta)\to (-\mu, -\zeta)$. The Gibbons-Hawking ansatz with
\[
h= A+ \frac{2m-4}{2r}, \quad r^2= |\mu|^2+|\zeta|^2, \quad A=\text{const}>0, \quad m=0,1,2,\ldots
\]
describes the asymptote at infinity of type $D_m$ ALF gravitational instantons.
The special case $m=0$ is known as the Atiyah-Hitchin metric, and $m=1$ is known as the Dancer metrics. The crucial difference with the $A_k$ case is that the $S^1$-symmetry is not global but only asymptotic in the $D_m$ case. A good survey for the Gibbons-Hawking ansatz and gravitational instantons is \cite[section 3]{Foscolo}.

\end{eg}

\begin{eg}\label{Calabiansatz}
Now move to complex dimension 3.
Let $\omega_{D,CY}$ be a Calabi-Yau metric on $D$ with $\omega_{D,CY}^2= \frac{1}{2}\Omega_D\wedge \overline{\Omega}_D$.
The special solution
\[
\begin{cases}
\tilde{\omega}=c \mu \omega_{D,CY},
\\
h=c^2\mu^2,
\\
d\vartheta= -c\omega_{D,CY}
\end{cases}
\]
corresponds to the Calabi ansatz. See section \ref{TianYau} for another perspective.

\end{eg}

The explicit conversion from the Calabi-Yau case to the Apostolov-Salamon setting is as follows. The data of $h, \vartheta$ are of course the same in both settings. We add the flat 7-th dimension parametrized by the coordinate $\tau$. Let the $SU(3)$-structure be
\[
\begin{cases}
\omega= d\mu \wedge d\tau+ \text{Im}\Omega_D,
\\
h^{1/4}\Omega= (h^{-1/2}d\tau-\sqrt{-1}d\mu)(\tilde{\omega}+\sqrt{-1}h^{1/2}\text{Re}\Omega_D),
\end{cases}
\]
so the $G_2$-structure is
\[
\begin{cases}
\phi= \omega' d\tau+ \text{Re}\Omega',\\
*_\phi\phi= -d\tau\wedge \text{Im}\Omega'+ \frac{1}{2}\omega'^2,
\end{cases}
\]
satisfying the Apostolov-Salamon equation before the rescaling.

To make a more direct comparison with the collapsing setup we also record the rescaled version: write the Calabi-Yau structure as
\[
\begin{cases}
\omega'= \epsilon \vartheta\wedge d\mu+\epsilon\tilde{\omega},
\\
\Omega'=\epsilon (hd\mu- \sqrt{-1}\epsilon\vartheta)\wedge \Omega_D.
\end{cases}
\]
The family of K\"ahler metrics satisfy
\[
\begin{cases}
\partial^2_\mu \tilde{\omega}+ \epsilon^{-1} d_Dd_D^c h=0 \quad \text{on $D$},
\\
d\vartheta= -\partial_\mu \tilde{\omega}- \epsilon^{-1}d_D^c h \wedge d\mu,
\\
\tilde{\omega}^2= \frac{1}{2} h\Omega_D\wedge \overline{\Omega}_D.
\end{cases}
\]
To convert this to the $SU(3)$-structure on $M/S^1$, write
\begin{equation}\label{conversionformularescaled}
\begin{cases}
\omega= d\mu \wedge d\tau+ \epsilon\text{Im}\Omega_D,
\\
h^{1/4}\Omega= \epsilon (h^{-1/2}d\tau-\sqrt{-1}d\mu)(\tilde{\omega}+\sqrt{-1}h^{1/2}\text{Re}\Omega_D),
\end{cases}
\end{equation}
so upon substituting
\[
\begin{cases}
\phi= \epsilon \vartheta \wedge \omega+ h^{3/4} \text{Re}\Omega,
\\
*_\phi \phi= -\epsilon h^{1/4} \vartheta \wedge \text{Im}\Omega + \frac{1}{2}h\omega^2,
\end{cases}
\]
compatibly with the rescaled convention (\ref{ASrescaledconvention}), then
\[
\begin{cases}
\phi= \omega' d\tau+ \text{Re}\Omega',\\
*_\phi\phi= -d\tau\wedge \text{Im}\Omega'+ \frac{1}{2}\omega'^2.
\end{cases}
\]
Thus the Calabi-Yau condition will lead to the rescaled Apostolov-Salamon equation (\ref{AScollapsing}). In particular, the complex 3-dimensional case of the degenerate Calabi-Yau metrics studied in \cite{SunZhang} can be viewed as a dimensionally reduced example of our proposal.

It is instructive to reexamine the polarized positive section $H$ in this dimensionally reduced setting. We have
\[
\frac{\partial H}{\partial \tau}=[\text{Re}\Omega_D]=const, \quad \frac{\partial H}{\partial  \mu}=-[\tilde{\omega}], \quad [\omega_y]=[\text{Im}\Omega_D].
\]
Combined with $[d\vartheta]=-\partial_\mu [\tilde{\omega}]$, we see $H$ can only be affine linear in $\tau$ and quadratic in $\mu$. A special instance is the Calabi ansatz in example \ref{Calabiansatz}, where
\[
H= [\text{Re}\Omega_D]\tau + \frac{1}{2} [d\vartheta]\mu^2.
\]
The fact that such examples are so restrictive suggests that the landscape of 7-dimensional examples is not well captured by the dimensional reduction. Rather the merit of the Calabi-Yau case is to guide our speculations about how to compactify the collapsing $G_2$-manifolds.

\section{Formal power series solution}\label{Formalpowerseriessolution}

The goal of this section is to show that over a given local base $S$ without singular fibres, the approximate solution introduced in section \ref{Slowcollapsing} can be perturbed into a formal power series solution to the Apostolov-Salamon equation (\ref{AScollapsing}):
\[
\begin{cases}
& h_\epsilon= h+ \epsilon h^{(1)}+ \epsilon^2 h^{(2)}+\ldots,
\\
& \vartheta_\epsilon= \vartheta+ \epsilon \theta^{(2)}+ \epsilon^2 \theta^{(3)}+\ldots,
\\
& \omega_\epsilon= \omega_S+ \epsilon \omega_y+ \epsilon^2 \omega^{(1)}+ \epsilon^3 \omega^{(2)}+\ldots,
\\
& h_\epsilon^{1/4}\text{Im}\Omega_\epsilon= \epsilon h^{1/4} (\omega_1dy_1+\omega_2dy_2) +\epsilon^2 \rho^{(1)}+ \epsilon^3 \rho^{(2)}+\ldots,
\\
& H_\epsilon= H+ \epsilon H^{(1)}+ \epsilon^2 H^{(2)}+ \ldots.
\end{cases}
\]

\begin{rmk}
Here $H, h, \vartheta, \omega_S, \omega_y, \omega_1,\omega_2$ together with the horizontal distribution are constructed in section \ref{Slowcollapsing}. The superscripts suggest the number of iterations in an inductive construction. The $\theta^{(k)}$ terms are 1-forms instead of connections. The term $h^{1/4}\text{Im}\Omega$ is viewed as an independent variable, in favour of $\text{Im}\Omega$, because we wish to keep this term $d$-closed. The power series solution is only formal because the inductive steps involve higher derivatives in the horizontal directions, coupled to higher powers of $\epsilon$. However, this whole discussion presumes all K3 fibres are smooth, so all functions involved are $C^\infty$.	
\end{rmk}

The key feature of the Apostolov-Salamon system is that it contains the Calabi-Yau monopole system which (almost) decouples from the rest of the equations (\cf section \ref{CalabiYaumonopolesystem}). Then given the solution to the Calabi-Yau monopole system, one tries to solve for $\omega_\epsilon, h_\epsilon^{1/4}\text{Im}\Omega_\epsilon$ as an overdetermined system, where the existence of solution crucially relies on the Calabi-Yau monopole system as an integrability condition.



\subsection{Gauge fixing issues}\label{gaugefixingissues}

While in the formal limit the K3 surfaces are `complex submanifolds' in the sense that $\Omega|_{K3}=0$ (equivalently $\text{Im}\Omega|_{K3}=0$), this can not be expected to hold in the finite $\epsilon$ setting. To explain this heuristically, given that $d(h_\epsilon^{1/4}\text{Im}\Omega_\epsilon)=0$, we wish to deform the K3 fibres such that $h_\epsilon^{1/4}\text{Im}\Omega_\epsilon|_{K3}=0$. The closedness condition enables us to write fibrewise
\[
h_\epsilon^{1/4} \text{Im}\Omega_\epsilon|_{K3} = df_1\wedge \omega_1+ df_2\wedge \omega_2+ df_3\wedge \omega_y,
\] 
and deformation of the fibre along a vector field $v$ gives a first order correction
\[
\mathcal{L}_v (h_\epsilon^{1/4} \text{Im}\Omega_\epsilon  )= d(\iota_v (h_\epsilon^{1/4} \text{Im}\Omega_\epsilon  ) )\approx d( a \omega_1+ b\omega_2 )
\]
for suitable functions $a, b$ corresponding to the components of $v$. This can be used to cancel the $df_1, df_2$ terms in $h_\epsilon^{1/4} \text{Im}\Omega_\epsilon|_{K3}$, but insufficient to remove the $df_3$ term. From this discussion, we see it is reasonable to impose 
\begin{equation}\label{pseudocalibration}
h_\epsilon^{1/4} \text{Im}\Omega_\epsilon|_{K3}= df\wedge \omega_y
\end{equation}
for some function $f$, as a compatibility condition between the K3 fibration structure and the Apostolov-Salamon equation, which is a  conceptual substitute for the calibration condition. Its chief effect is to rigidify the fibration.

The $S^1$-bundle causes a gauge ambiguity for $\vartheta_\epsilon$, but this can be largely ignored because only the curvature $d\vartheta_\epsilon$ enters into the Apostolov-Salamon equation, and gauge equivalent choices give isomorphic constructions.

Next, we seek an optimal way to represent the Apostolov-Salamon solution with respect to the diffeomorphism action preserving the fibration structure. Given the K3 fibration structure, and the condition $d\omega_\epsilon=0$, with fixed class $\epsilon^{-1}[\omega_\epsilon]$, the standard Moser's trick allows us to assume that  $\epsilon^{-1}\omega_\epsilon|_{K3}$ is fixed to be $\omega_y$. This reduces the fibrewise diffeomorphism group to the fibrewise symplectic group; to remove this, we can impose further that fibrewise
\begin{equation}
(\omega_1dy_1+\omega_2dy_2)\wedge h_\epsilon^{1/4} \text{Im}\Omega_\epsilon= \text{const}(y,\epsilon) \omega_y^2\wedge \omega_S.
\end{equation}

We have thus removed the fibrewise diffeomorphism group. The remaining gauge freedom is the fibration preserving diffeomorphisms moving the fibres around, alternatively thought as the diffeomorphisms of $S$, which are symmetries of the weighted maximal submanifold equation. We first associate a polarized positive section $H_\epsilon$ to a solution of the Apostolov-Salamon system. Notice that modulo $d(f\omega_y)$, the closed form $h_\epsilon^{1/4}\text{Im}\Omega_\epsilon$ vanishes on K3 fibres, so defines a closed $H^2(K3)$-valued 1-form on $S$. If we impose $H^1(S)=0$, or if we demand all $\rho^{(k)}$ to be exact (as we will later do), then this is exact, so can be written as $dH_\epsilon$ for some $H_\epsilon: S\to H^2(K3)$. Now if $f$ is chosen appropriately, then $H_\epsilon$ can be made orthogonal to $[\omega_y]$ pointwise on $S$. Morever, since $H_\epsilon$ is well approximated by $H$, it is also a positive section. Our gauge fixing condition is that $H_\epsilon$ differs from $H$ by a normal vector field to the image of $H$.
This makes $H_\epsilon$ canonically defined.

\begin{rmk}
The gauge fixing is convenient for our induction scheme, but the condition (\ref{pseudocalibration}) seems less fundamental compared to the coassociative K3 fibration analgoue. 
The iterated fibration picture relies on the $S^1$-symmetry reduction assumption, which cannot persist on nontrivial compact examples. Instead, it is expected to be a good effective description up to exponentially suppressed error, in the generic region of compact examples near the adiabatic limit. 

\end{rmk}

\subsection{Iterative scheme setup}

We write 
\[
\begin{cases}
& h^{[k]}= h+ \epsilon h^{(1)}+ \epsilon^2 h^{(2)}+\ldots+ \epsilon^k h^{(k)} ,
\\
& \vartheta^{[k]}= \vartheta+ \epsilon \theta^{(2)}+ \epsilon^2 \theta^{(3)}+\ldots+ \epsilon^{k-1}\theta^{(k)},
\\
& \omega^{[k]}= \omega_S+ \epsilon \omega_y+ \epsilon^2 \omega^{(1)}+ \epsilon^3 \omega^{(2)}+\ldots+ \epsilon^{(k+1)}\omega^{(k)},
\\
& (h^{1/4}\text{Im}\Omega)^{[k]}= \epsilon h^{1/4} (\omega_1dy_1+\omega_2dy_2) +\epsilon^2 \rho^{(1)}+ \epsilon^3 \rho^{(2)}+\ldots+ \epsilon^{(k+1)}\rho^{(k)},
\\
& H^{[k]}= H+ \epsilon H^{(1)}+ \epsilon^2 H^{(2)}+ \ldots+ \epsilon^kH^{(k)}.
\end{cases}
\]
These are best thought as representatives of formal power series modulo higher powers of $\epsilon$.
They are required to satisfy the \textbf{approximate Apostolov-Salamon equation}
\begin{equation}
\begin{cases}
& d\omega^{[k]}=0, \\
&\frac{1}{2} dh^{[k]}\wedge (\omega^{[k]})^2= \epsilon  d\vartheta^{[k]} \wedge (h^{1/4}\text{Im}\Omega)^{[k]}+ O(\epsilon^{k+2}),
\\
& d (h^{3/4}\text{Re}\Omega)^{[k]} =  - \epsilon d\vartheta^{[k]} \wedge \omega^{[k]}+ O( \epsilon^{k+2} ),
\\
&
d(h^{1/4}\text{Im}\Omega)^{[k]}= 0,
\end{cases}
\end{equation}
and the \textbf{approximate $SU(3)$-structure} constraints
\begin{equation}
\begin{cases}
\omega^{[k]}\wedge \text{Im}\Omega^{[k]}= O(\epsilon^{k+3}), 
\\
(\omega^{[k]})^3= \frac{3}{2} \text{Re}\Omega^{[k]}\wedge \text{Im}\Omega^{[k]} (1+ O(\epsilon^{k+1}) )
\end{cases}
\end{equation}
Here $\text{Re}\Omega^{[k]}$ is canonically determined by $\text{Im}\Omega^{[k]}$ so that $\Omega^{[k]}$ is a complex volume form.
Notice we require at the onset that $\omega^{[k]}$ and $(h^{1/4}\text{Im}\Omega)^{[k]}$ remain closed.

The following version of the \textbf{gauge fixing condition} will be used:
\[
\begin{cases}
\epsilon^{-1}\omega^{[k]}|_{K3}=\omega_y,
\\
(h^{1/4}\text{Im}\Omega)^{[k]}|_{K3}=\epsilon^2d(\text{function})\wedge \omega_y
,
\\
(\omega_1dy_1+\omega_2dy_2)\wedge (h^{1/4}\text{Im}\Omega)^{[k]}= \epsilon^2 (\text{function of $y,\epsilon$}+ O(\epsilon^k) ) \omega_y^2\wedge \omega_S,
\\
H^{[k]}-H \perp [\omega_1], [\omega_2],[\omega_y].
\end{cases}
\]
 We shall also make use of a number of  \textbf{normalisation conditions} during the induction below. For instance, we will require 
\begin{itemize}
\item The correction terms $\omega^{(k)}$ and $\rho^{(k)}$ are exact for all $k$.
\item The boundary value of $H^{[k]}$ on $\partial S$ is fixed to be $H$.
\item The total volume $\int (h^{[k]})^{1/2} (\omega^{[k]})^3= 3\epsilon^2 \int_S h^{1/2} \omega_S(1+O(\epsilon^{k+1}))$.
\end{itemize}
Another more technical auxiliary condition will appear in the inductive hypothesis in section \ref{CalabiYaumonopolesystem}.

\begin{rmk}
By the $O(\epsilon^k)$ notation, we mean $\epsilon^k$ times a formal power series in $\epsilon$ of smooth forms. In particular, if a quantity is $O(\epsilon^k)$, then so are its derivatives in this convention.
\end{rmk}


\begin{thm}
Given a solution of the weighted maximal submanifold equation $H$ on a surface $S$ with boundary, avoiding excess $(-2)$-classes. Then there exists a unique formal power series solution $(h_\epsilon, \vartheta_\epsilon, \omega_\epsilon, \Omega_\epsilon)$ solving the Apostolov-Salamon equation, the $SU(3)$-structure condition, the gauge fixing conditions, and the various normalisation conditions. 
\end{thm}

In the sequel we will concentrate on the existence part; the uniqueness is more or less a byproduct.
We shall assume the conditions are satisfied up to $k-1$, and the induction step is roughly to solve the equations in step $k$. The actual order in which the equations are solved is a little more intricate, due to the fact that the hierarchy of equations is not only organized in order of $\epsilon$ but also according to the filtration. That is, when we solve for $h^{(k)}, \vartheta^{(k)}, \omega^{(k)}, \rho^{(k)}, H^{(k)}$, we assume $h^{[k-1]}$ etc to be fixed by induction, except that $\rho^{(k-1)}$ can still be modified by $\omega_S\wedge d(\text{function})$, and $\omega^{(k-1)}$ can still be modified by the pullback of a 2-form from $S$.
We shall describe how to solve the system from an algorithmic viewpoint, namely we repeatedly update the choices to satisfy more and more conditions, thus making the error functions take increasingly special forms, often keeping the same notation.

\subsection{Calabi-Yau monopole system}
\label{CalabiYaumonopolesystem}

The Calabi-Yau monopole system refers to the integrability conditions
\begin{equation}\label{CalabiYaumonopolekthstep}
\begin{cases}
d\vartheta^{[k]} \wedge (\omega^{[k-1]})^2= O(\epsilon^{k+1}),\\
\frac{1}{2} dh^{[k]}\wedge (\omega^{[k-1]})^2=\epsilon  d\vartheta^{[k]} \wedge (h^{1/4}\text{Im}\Omega)^{[k-1]} +O(\epsilon^{k+2}).
\end{cases}
\end{equation}
The word `Calabi-Yau' comes from the analogy with \cite{HaskinsS1}, even though in our setting the 6-manifold needs not be Calabi-Yau. Rewriting the equations in terms of the unknown functions $\theta^{(k)}, h^{(k)}$,
\[
\begin{cases}
&2d\theta^{(k)} \wedge \omega_y \wedge \omega_S = - \epsilon^{-k} d\vartheta^{[k-1]} \wedge (\omega^{[k-1]})^2+ O(\epsilon)= \alpha^{(k)} +O(\epsilon)
\\
& dh^{(k)}\wedge \omega_y\wedge \omega_S- d\theta^{(k)}\wedge (\omega_1dy_1+\omega_2dy_2) 
\\
& = \epsilon^{-k-1}(\frac{1}{2} dh^{[k-1]}\wedge (\omega^{[k-1]})^2-\epsilon  d\vartheta^{[k-1]} \wedge (h^{1/4}\text{Im}\Omega)^{[k-1]}) +O(\epsilon)=\beta^{(k)}+ O(\epsilon).
\end{cases}
\]
From the induction hypothesis, the  $\alpha^{(k)}, \beta^{(k)}$ are closed forms of order $O(1)$, and furthermore we assume the following cohomological requirement by induction
\begin{itemize}
\item $\alpha^{(k)}=\omega_S \wedge $(exact 4-form on K3 fibres),
\item $\beta^{(k)}=dy_1\wedge$(exact 4-form on K3 fibres)$+dy_2$(exact 4-form on K3 fibres) modulo $dy_1\wedge dy_2$.
\end{itemize}
We determine the restriction of $d\theta^{(k)}$ to K3 fibres by prescribing $d\theta^{(k)}\wedge \omega_1, d\theta^{(k)}\wedge \omega_2$, and $d\theta^{(k)}\wedge \omega_y$, which amounts to prescribing the self dual part of an exact 2-form on the hyperk\"ahler K3, and can be done precisely by the cohomological condition on $\alpha^{(k)},\beta^{(k)}$.
Notice that the ambiguity in $\rho^{(k-1)}, \omega^{(k-1)}$ so far plays no role.

After this step we have reduced to $\alpha^{(k)}=0$ and $\beta^{(k)}=dy_1dy_2\wedge$(closed 3-form on K3 fibres). The modification of $\rho^{(k-1)}$ takes place immediately after this, using the rest of the Apostolov-Salamon system, see section \ref{DeformationSU3I} below. Granted that $\rho^{(k-1)}$ is fixed, we proceed with the rest of the Calabi-Yau monopole system. Writing the closed 3-forms as
\[
df_1\wedge \omega_1+ df_2\wedge \omega_2+ df_3\wedge \omega_y,
\]
we see that the $df_1, df_2$ terms can be cancelled by modifying $\theta^{(k)}$ by $ady_1+bdy_2$ depending linearly on $f_1, f_2$, and $df_3$ can be cancelled by adjusting $h^{(k)}$. We have thus satisfied (\ref{CalabiYaumonopolekthstep}), and the remaining ambiguity is
\begin{itemize}
\item $h^{(k)}$ is determined up to a fibrewise constant,
\item $\theta^{(k)}$ is determined up to the pullback of a 2-form on $S$.
\end{itemize}
We shall see that these are fixed by cohomological requirements, up to a global constant on $h^{(k)}$. Notice the ambiguity of $\omega^{(k-1)}$ has not played any role.

Notice that because of the smallness of the correction induced by the yet unknown $\omega^{(k)}, \rho^{(k)}$,   (\ref{CalabiYaumonopolekthstep}) is equivalent to
\[
\begin{cases}
d\vartheta^{[k]} \wedge (\omega^{[k]})^2= O(\epsilon^{k+1}),\\
\frac{1}{2} dh^{[k]}\wedge (\omega^{[k]})^2=\epsilon  d\vartheta^{[k]} \wedge (h^{1/4}\text{Im}\Omega)^{[k]} +O(\epsilon^{k+2}).
\end{cases}
\]
This is of course the statement that $\alpha^{(k+1)}, \beta^{(k+1)}$ are of order $O(1)$. It is not sensitive to the ambiguity of $\omega^{(k-1)}$ caused by the pullback of a 2-form on $S$. The inductive construction of $\omega^{[k]}, (h^{1/4}\text{Im}\Omega)^{[k]}$ guarantees this is exact. Our induction relies on the cohomological requirement that
\begin{itemize}
	\item $\alpha^{(k+1)}=\omega_S \wedge $(exact 4-form on K3 fibres),
	\item $\beta^{(k+1)}=dy_1\wedge$(exact 4-form on K3 fibres)$+dy_2$(exact 4-form on K3 fibres) modulo $dy_1\wedge dy_2$.
\end{itemize}
These exactness conditions amount to conditions on fibrewise integrals. This requirement on $\alpha^{(k+1)}$ can be achieved precisely by adjusting $\theta^{(k)}$ by a 2-form pulled back from $S$. As for $\beta^{(k)}$, it defines an exact $H^4(K3)$-valued 1-form on $S$. Now adjusting $h^{(k)}$ by a function of $y$ we can cancel this exact 1-form, thus achieving the cohomological requirement.
Later in our construction, this cohomological condition is not sensitive to corrections caused by $\omega^{(k)}$, but $\rho^{(k)}$ will damage it in a very specific way. We shall later return to this issue and correct $h^{(k)}$ by a further function depending only on $y\in S$
(\cf section \ref{maximalsubmanifoldsystem}).

The above only concerns $dh^{(k)}$, rather than $h^{(k)}$ itself.
The global constant for $h^{(k)}$ is specified by the total volume normalisation
\[
\int h^{[k]} (\omega^{[k-1]})^3= 3\epsilon^2 \int_S h \omega_S(1+O(\epsilon^{k+1})).
\]
This condition is not sensitive to $\omega^{(k)}$. The place to take care of this global constant is in section \ref{DeformationSU3II}. 


\begin{rmk}
The above discussion is uniform for $k\geq 2$. For $k=1$ this story is basically the same as the determination of $\vartheta$ we saw in section \ref{Slowcollapsing}, the small difference being that $\vartheta$ is a connection while the higher corrections are 1-forms, and the role of the cohomological requirement for $k=1$ is instead played by the prescription of the first Chern class $\frac{1}{2\pi}[d\vartheta]$.
\end{rmk}

\begin{rmk}\label{CYmonopolebetter}
Once the algorithm fixes $\omega^{(k-1)}$ evantually, we can keep $h^{[k]}$ fixed and make an adjustment to $\theta^{(k)}$ with a further $\epsilon$ power in front. By one more iteration of some steps above, we can thus ensure a better condition
\[
d\vartheta^{[k]}\wedge (\omega^{[k-1]})^2 =O(\epsilon^{k+2} ),
\]
while keeping all others.
Notice that independent of the unknown $\omega^{(k)}$, we will still have
\[
d\vartheta^{[k]}\wedge (\omega^{[k]})^2=O(\epsilon^{k+2}),
\]
using that $\omega^{(k)}$ is required to lie in the filtration with at least one $dy_i$ term.
\end{rmk}

\subsection{Deformation of the $SU(3)$-structure I: first filtration}\label{DeformationSU3I}

The rest of the Apostolov-Salamon system that remains to be solved involves
\[
d (h^{3/4}\text{Re}\Omega)^{[k]} =  - \epsilon d\vartheta^{[k]} \wedge \omega^{[k]}+ O( \epsilon^{k+2} ),
\]
together with the $SU(3)$-structure requirements, and the gauge fixing conditions. Essentially the degrees of freedom at our disposal is to vary $\omega^{[k]}$ keeping the closedness condition and the prescribed restriction to K3 fibres, and to vary $(h^{1/4}\text{Im}\Omega)^{[k]}$ keeping the closedness condition. We shall see the problem is rather overdetermined. Our strategy involves solving the equation according to the Leray filtration determined by the K3 fibration structure. In this section we discuss the first filtred piece of the equations, which is tied up with the issue of modifying $\rho^{(k-1)}$. This step takes place after we have specified $\theta^{(k)}|_{K3}$, before we solve the rest of the Calabi-Yau monopole system. 

\begin{rmk}
In the discussion below 
$\rho^{(k)}$ plays a much more major role than $\omega^{(k)}$, which does not affect most equations. The reader can also keep in mind that the increment $\epsilon^{k+1}\omega^{(k)}$ will not be added until almost the end of the algorithm.
\end{rmk}

We write 
\[
\epsilon^{-k-1}(d (h^{3/4}\text{Re}\Omega)^{[k]} + \epsilon d\vartheta^{[k]} \wedge \omega^{[k]})=h^{1/2}d(h^{1/4}\text{Re}\Omega )^{(k)} + \gamma_k+ O( \epsilon )
\]
where $\epsilon^{k+1}(h^{1/4}\text{Re}\Omega )^{(k)}$ denotes the small increment to $(h^{1/4}\text{Re}\Omega)^{[k]}$ induced by the small increment $\epsilon^{k+1}\rho^{(k)}$ for $(h^{1/4}\text{Im}\Omega)^{[k]}$, and
\[
\gamma_k=\epsilon^{-k-1}(d (h^{3/4}\text{Re}\Omega)^{[k-1]} + \epsilon d\vartheta^{[k]} \wedge \omega^{[k-1]}).
\]
Clearly $\gamma_k$ is closed.  Morever, restricted to the K3 fibres, 
\[
[\gamma_k]= \epsilon^{-k}[d\vartheta]\wedge [\omega_y]=0\in H^4(K3).
\]
Notice the restriction of $\gamma_k$ to K3 fibres depends only on $\theta^{(k)}$ restricted to the K3 fibres, not the full solution to the Calabi-Yau monopole system.
We are allowed to modify $\rho^{(k)}$ by a term proportional to $d(f\omega_y)$ compatible with the gauge condition (\ref{pseudocalibration}) and the closedness of $(h^{1/4}\text{Im}\Omega)^{[k]}$, where $f$ is a function to be determined. Modulo higher order in $\epsilon$ and restricted to K3 fibres, the corresponding change to $(h^{1/4}\text{Re}\Omega)^{(k)}$ is $-J df\wedge \omega_y$, where $J$ is the complex structure on the K3 fibres. We recognize the fibrewise equation
\[
d( Jdf\wedge \omega_y )=\gamma_k|_{K3}
\]
is just the Laplace equation, so by the cohomological condition can be solved fibrewise up to fibrewise constants. After this step, we can assume $\gamma_k|_{K3}=0$. To fix $f$ as a function of $y$, we notice $\omega^{[k-1]}\wedge (h^{1/4}\text{Im}\Omega)^{[k-1]}$ is an exact 5-form so defines an exact $H^4(K3)$-valued 1-form, hence can be made to vanish precisely by utilizing $f$ as a function of $y$. This determines $f$ up to a global constant, which does not affect $\rho^{(k)}$.

The subtle issue is that while by induction
\[
\omega^{[k-1]}\wedge (h^{1/4}\text{Im}\Omega)^{[k-1]}= O(\epsilon^{k+2}),
\]
the term $df\wedge \omega_y$ in $\rho^{(k)}$ contributes to $\omega^{[k]}\wedge (h^{1/4}\text{Im}\Omega)^{[k]}$ by
$\epsilon^{k+1}df\wedge \omega_y\wedge \omega_S$. This needs to be compensated by modifying $\rho^{(k-1)}$ by $-d(f\omega_S)$. This fixes the little ambiguity of $\rho^{(k-1)}$. In the algorithm, we return to the Calabi-Yau monopole system and solve for $h^{(k)}, \theta^{(k)}$. In the remaining discussions, we shall take $h^{[k]}, \vartheta^{[k]}$ as fixed except when announced otherwise, and we have
\[
\omega^{[k]}\wedge \text{Im}\Omega^{[k]}= O(\epsilon^{k+2}),
\]
which we seek to improve.

\subsection{Deformation of the $SU(3)$-structure II: second filtration}
\label{DeformationSU3II}

Next we consider the filtration with $dy_1$ or $dy_2$ terms. We are allowed to modify $\rho^{(k)}$ by a term proportional to 
\[
dy_1 \wedge d(\text{1-form on K3})+ dy_2 \wedge d(\text{1-form on K3}).
\]
An efficient way to prescribe an exact 2-form on K3 fibres is to prescribe its self dual part:
\[
f_{11}\omega_1+ f_{12}\omega_2+f_{13}\omega_y, \quad f_{21}\omega_1+ f_{22}\omega_2+f_{23}\omega_y
\]
where $\int_{K3} f_{ij}\omega_y^2=0$. We write the increment of $\rho^{(k)}$ modulo $dy_1\wedge dy_2$ term as
\[
dy_1 \wedge (f_{11}\omega_1+ f_{12}\omega_2+f_{13}\omega_y+ASD_1)   + dy_2 \wedge (f_{21}\omega_1+ f_{22}\omega_2+f_{23}\omega_y+ASD_2)
\]
Working at a given fibre in the coodinates with $g_{ij}=\delta_{ij}$,
the corresponding change to $(h^{1/4}\text{Re}\Omega)^{(k)}$ modulo $dy_1dy_2$ term is
\[
dy_1 (ASD_2+f_{23}\omega_y+f_{11}\omega_2-f_{12}\omega_1) + dy_2(-ASD_1-f_{13}\omega_y+f_{21}\omega_2-f_{22}\omega_1  ).
\]
Using $d\rho^{(k)}=0$ to eliminate the ASD terms, the change to $d(h^{1/4}\text{Re}\Omega)^{(k)}$ modulo $dy_1dy_2$ term is
\[
d(f_{11}-f_{22})\wedge (\omega_2dy_1+\omega_1dy_2)+ d(f_{12}+f_{21}) \wedge (\omega_2dy_2-\omega_1dy_1).
\]
Notice the answer depends only on two combination of the 6 functions $f_{ij}$, an indication that integrability is at play.

The roles of the other 4 functions are as follows:
\begin{itemize}
\item The horizontal-vertical $(1,4)$ component of $\omega^{[k]}\wedge \text{Im}\Omega^{[k]}$ is controlled precisely by $f_{13}, f_{23}$. 
By adjusting $f_{13}, f_{23}$ we can make $\omega^{[k]}\wedge (h^{1/4}\text{Im}\Omega)^{[k]}=O(\epsilon^{k+3})$ modulo $dy_1\wedge dy_2$ term.
Here notice that the yet unknown $\omega^{(k)}$ has no effect to $O(\epsilon^{k+2})$ order due to our choice to fix $\omega_y$ as the restriction to K3 fibres. Notice also that the condition $\int_{K3} f_{i3}\omega_y^2=0$ is precisely compatible with our previous cohomological normalisation on $\omega^{[k-1]}\wedge (h^{1/4}\text{Im}\Omega)^{[k-1]}$.

\item
The volume normalisation on $\text{Re}\Omega^{[k]}\wedge \text{Im}\Omega^{[k]}$ is controlled by $f_{11}+f_{22}$. By possibly modifying $\omega^{(k-1)}$ by the pullback of a 2-form on $S$, which does not affect the inductive hypotheses, we can ensure that
\[
(h^{[k]})^{1/2}(\omega^{[k-1]})^3= \frac{3}{2}(1+O(\epsilon^k)) (h^{1/4}\text{Re}\Omega)^{[k-1]}\wedge (h^{1/4}\text{Im}\Omega)^{[k-1]},
\]
where the $O(\epsilon^k)$ function has \emph{fibrewise integral zero} with respect to $\omega_y^2$. By allowing $h^{(k)}$ to drift by a global constant, accompanied by a corresponding change in $\omega^{(k-1)}$ so as to preserve the fibrewise integral zero condition, we can enforce the total volume normalisation on
\[
\int h^{[k]} (\omega^{[k-1]})^3
\] 
as promised in section \ref{CalabiYaumonopolesystem}.

Now the $f_{11}+f_{12}$ term gives rise to a term proportional to $\text{Im}\Omega$ in $\rho^{(k)}$, whose effect is to modify $(h^{1/4}\text{Re}\Omega)^{[k]}\wedge (h^{1/4}\text{Im}\Omega)^{[k]}$, and we can use this to cancel the $O(\epsilon^k)$ term. This ensures
\[
(\omega^{[k]})^3= \frac{3}{2}(1+O(\epsilon^{k+1})) \text{Re}\Omega^{[k]}\wedge \text{Im}\Omega^{[k]},
\]
since $\omega^{(k)}$ does not give rise to corrections to volume of relative order $O(\epsilon^k)$. Notice again that we need the cohomological normalisation to be compatible with $\int_{K3} f_{ij}\omega_y^2=0$.

\item
The combination $f_{12}-f_{21}$ controls the gauge fixing condition on $(h^{1/4}\text{Im}\Omega)^{[k]}\wedge (\omega_1dy_1+\omega_2dy_2)$. It almost renders the expression zero modulo $O(\epsilon^{k+2})$, except that the integral normalisation $\int_{K3} f_{ij}\omega_y^2=0$ means that we cannot remove a fibrewise constant depending on $y\in S$.

\end{itemize}

To summarize, while keeping the closedness of $(h^{1/4}\text{Im}\Omega)^{[k]}$, there are six free functions at our disposal, which we can use to achieve three constraints in the $SU(3)$-structure condition, one gauge fixing condition, and the two others can be used to adjust $\gamma_k$ by 
\[
da\wedge (\omega_2dy_1+\omega_1dy_2)+ db\wedge (\omega_2dy_2-\omega_1dy_1) \quad \mod dy_1\wedge dy_2 
\]
for any two prescribed function $a, b$ on the fibres.

\subsection{Deformation of the $SU(3)$-structure III: integrability}
\label{DeformationSU3III}

The crucial feature of the Apostolov-Salamon system is that while the equations governing the deformation of $h^{1/4}\text{Im}\Omega$ are quite overdetermined, there are additional integrability from the Calabi-Yau monopole system that comes to the rescue. To explain the key ideas, we will present the argument without keeping track of the power of $\epsilon$ error everywhere.

We begin with some basic linear algebra of $SU(3)$-structures:

\begin{lem}
Let $(\omega, \Omega)$ be an $SU(3)$-structure, and $v$ be a vector field. Then
\[
\begin{cases}
\omega^2\wedge \iota_v \omega= - \text{Re} \Omega \wedge \iota_v \text{Im}\Omega,
\\
\iota_v \text{Re}\Omega=\iota_{Jv} \text{Im}\Omega,
\\
\text{Im}\Omega\wedge \iota_{Jv}\omega= \omega\wedge \iota_{v} \text{Re}\Omega,
\\
\iota_v d\text{Im}\Omega \wedge \text{Re}\Omega =\iota_v \text{Im}\Omega \wedge d\text{Re}\Omega
\end{cases}
\]
Here the last identity comes from  Hitchin's stable 3-forms, and expresses the relation between the 6-dimensional components of $d\text{Re}\Omega, d\text{Im}\Omega$.
\end{lem}

\begin{prop}
Assume $(\omega_\epsilon, \Omega_\epsilon)$ is an $SU(3)$-structure, $h_\epsilon$ a positive valued function, and $d\vartheta_\epsilon$ a closed 2-form, satisfying
\[
\begin{cases}
\epsilon h_\epsilon^{1/4} d\vartheta_\epsilon \wedge \text{Im}\Omega_\epsilon= \frac{1}{2} dh_\epsilon \wedge \omega_\epsilon^2,
\\
d(h_\epsilon^{1/4}\text{Im}\Omega_\epsilon)=0,
\end{cases}
\]
then for any vector field $v$,
\[
( \epsilon d\vartheta_\epsilon \wedge \omega_\epsilon+ d(\text{Re}\Omega_\epsilon h_\epsilon^{3/4}  )  )\wedge \iota_v \text{Re}\Omega_\epsilon=0,
\]
namely the 6-dimensional component of $\epsilon d\vartheta_\epsilon \wedge \omega_\epsilon+ d(\text{Re}\Omega_\epsilon h_\epsilon^{3/4}  )$ vanishes automatically.
\end{prop}

\begin{proof}
We compute using the first equation
\[
\epsilon d\vartheta_\epsilon \wedge \omega_\epsilon \wedge \iota_v \text{Re}\Omega_\epsilon= \epsilon  d\vartheta_\epsilon \wedge \text{Im}\Omega_\epsilon\wedge \iota_{Jv} \omega=\frac{1}{2} h_\epsilon^{-1/4} dh_\epsilon\wedge \omega_\epsilon^2 \wedge \iota_{Jv}\omega_\epsilon.
\]
By the second equation and Hitchin's stable form identity applied to $h_\epsilon^{1/4}\text{Im}\Omega_\epsilon$,
\[
0=d(\text{Im}\Omega_\epsilon h_\epsilon^{1/4})\wedge \iota_{Jv}\text{Re}\Omega_\epsilon= -\iota_{Jv} \text{Im}\Omega_\epsilon \wedge d(\text{Re}\Omega_\epsilon h_\epsilon^{1/4})= -\iota_{v} \text{Re}\Omega_\epsilon \wedge d(\text{Re}\Omega_\epsilon h_\epsilon^{1/4}) , 
\]
whence
\[
d(\text{Re}\Omega_\epsilon h_\epsilon^{3/4}) \wedge \iota_v \text{Re}\Omega_\epsilon= \frac{1}{2}h_\epsilon^{-1/4} dh_\epsilon\wedge \text{Re}\Omega_\epsilon\wedge \iota_{v}\text{Re}\Omega_\epsilon.
\]
The claim then follows from
\[
\omega_\epsilon^2 \wedge \iota_{Jv}\omega_\epsilon=- \text{Re}\Omega_\epsilon\wedge \iota_{Jv}\text{Im}\Omega_\epsilon= - \text{Re}\Omega_\epsilon\wedge \iota_{v}\text{Re}\Omega_\epsilon.
\]
\end{proof}

In our application, we use $h^{[k]}, \vartheta^{[k]}$, satisfying the Calabi-Yau monopole system (\ref{CalabiYaumonopolekthstep}). The $(h^{1/4}\text{Im}\Omega)^{[k]}$ term is always closed. The error to the $SU(3)$-structure, after our initial stages of correction, is
\[
\begin{cases}
(\omega^{[k]})^3= \frac{3}{2}(1+O(\epsilon^{k+1})) \text{Re}\Omega^{[k]} \wedge \text{Im}\Omega^{[k]},
\\
\omega^{[k]}\wedge \text{Im}\Omega^{[k]}=O(\epsilon^{k+3}) \quad \mod dy_1\wedge dy_2,
\\
\omega^{[k]}\wedge \text{Im}\Omega^{[k]}=O(\epsilon^{k+2})
\end{cases}
\]
Restoring the powers of $\epsilon$ into the above argument, we reach an important conclusion: modulo $O(\epsilon)$, the error function $\gamma_k$ satisfies
\begin{equation}\label{integrabilityCYmonopole}
\begin{cases}
\gamma_k\wedge \iota_v(\omega_1dy_1+ \omega_2dy_2)=0, \quad \text{for any vertical vector $v$},
\\
\gamma_k \wedge \omega_1=0,\quad \gamma_k\wedge \omega_2=0.
\end{cases}
\end{equation}
Now $d\gamma_k=0$ by construction, and at this stage $\gamma_k$ lies in the filtration with at least one $dy_i$ term, so modulo $dy_1dy_2$ terms we can write
\[
\gamma_k= (da_{11}\wedge \omega_1+da_{12}\wedge\omega_2+ da_{13}\omega_y)dy_1+ (da_{21}\wedge \omega_1+da_{22}\wedge\omega_2+ da_{23}\wedge\omega_y)dy_2
\]
for some functions $a_{ij}$ with fibre average zero. At a given fibre we work in the coordinate with $g_{ij}=\delta_{ij}$.  From the first part of the integrability condition (\ref{integrabilityCYmonopole}),
\[
(da_{11}\wedge \omega_1+da_{12}\wedge\omega_2+ da_{13}\wedge\omega_y)=J(da_{21}\wedge \omega_1+da_{22}\wedge\omega_2+ da_{23}\wedge\omega_y)
\]
on the K3 fibre. Taking $d$ on both sides gives $a_{23}=0$, so by symmetry $a_{13}=0$. Now
\[
da_{11}\wedge \omega_1+da_{12}\wedge\omega_2=J(da_{21}\wedge \omega_1+da_{22}\wedge\omega_2)=da_{21}\wedge \omega_2-da_{22}\wedge \omega_1 ,
\]
so by uniqueness of such a decomposition $a_{11}=-a_{22}$ and $a_{12}=a_{21}$. Thus
\[
\gamma_k= da_{11}\wedge (\omega_1dy_1-\omega_2dy_2)+ da_{12}(\omega_2dy_1- \omega_1dy_2) \quad \mod dy_1dy_2.
\]
But this is precisely of the form we can cancel by adjusting $\rho^{(k)}$, as discussed in section \ref{DeformationSU3II}. The conclusion is that we can now assume $\gamma_k$ is a multiple of $dy_1\wedge dy_2$.

\subsection{Weighted maximal submanifold system}\label{maximalsubmanifoldsystem}

At this stage 
\[
\gamma_k=dy_1\wedge dy_2\wedge (\text{closed 2-forms on K3})
\]
Thus $\gamma_k$ defines an $H^2(K3)$-valued 2-form on $S$, which we denote as $[\gamma_k]$. There is a natural decomposition of $[\gamma_k]$ into fibrewise self dual and ASD parts, or equivalently the tangential and normal component to the image of $H: S\to H^2(K3)$. Our next goal is to remove the ASD part of $[\gamma_k]$.

To correct this error, we reflect on the degrees of freedom not yet used. We have already adjusted $\rho^{(k)}$ by corrections of the form 
\[
dy_1(\text{exact 2-form on K3})+dy_2(\text{exact 2-form on K3})\quad \mod dy_1\wedge dy_2.
\]
However, the actual constraint is that $\rho^{(k)}$ should be exact, which leaves the possibility of adjusting by harmonic 2-forms in place of exact 2-forms, corresponding to the deformation of $H^{[k]}$ (\cf the last part of section \ref{gaugefixingissues}).
Given a vector field $H^{(k)}: S\to H^2(K3)$ normal to the weighted maximal submanifold defined by the image of $H$, we can view $\frac{\partial H^{(k)}}{\partial y_i}$ for $i=1,2$ as specifying the data of the families of harmonic 2-forms. The normality condition can be thought as a gauge choice, related to the diffeomorphism of $S$. The requirement that the two families of harmonic forms come from the gradient of a single $H^{(k)}$, means that we can arrange $\rho^{(k)}$ to be exact, instead of just being closed modulo $dy_1\wedge dy_2$.

One can then calculate the response of $(h^{1/4}\text{Re}\Omega)^{(k)}$. By reexamining the calculations about the exact forms in section \ref{DeformationSU3II}, one sees that the change in $d (h^{1/4}\text{Re}\Omega)^{(k)}$ lies in the filtration with $dy_1\wedge dy_2$ terms, so defines an $H^2(K3)$-valued 2-form on $S$. However, we need to be careful not to ruin the various conditions we already achieved. The fact that $H^{(k)}$ is orthogonal to the constant class $[\omega_y]$ implies that for the change in $\rho^{(k)}$ induced by $H^{(k)}$,
\[
\rho^{(k)}\wedge \omega_y=O(\epsilon), \quad \mod dy_1\wedge dy_2,
\]
so we still have
\[
\omega^{[k]}\wedge \text{Im}\Omega^{[k]}=O(\epsilon^{k+3})
\quad \mod dy_1\wedge dy_2.
\]
A more subtle effect happens to the Calabi-Yau monopole system. In section \ref{CalabiYaumonopolesystem} we required the cohomological condition that the $H^4(K3)$-valued 1-form on $S$ defined by $\beta^{(k+1)}$ vanishes. This forces us to adjust $h^{(k)}$ by a function depending only on $y$, satisfying an equation on $S$
\[
\frac{1}{2}dh^{(k)} [\omega_y]^2= [d\vartheta] \wedge dH^{(k)}.
\]
We take $h^{(k)}= 2[d\vartheta]\cdot H^{(k)}$ compatible with (\ref{hviaH}); the ultimate reason for this choice has to do with total volume normalisation, see below.
Thus the response of $d(h^{3/4}\text{Re}\Omega)^{(k)}$ is still in the $dy_1\wedge dy_2$ filtration, and defines an $H^2(K3)$-valued 2-form on $S$. This can be identified as follows: consider the map
\[
\tilde{H} \mapsto \frac{\partial }{\partial y_j} ( \tilde{h}^{1/2} \tilde{g}^{ij}\sqrt{\det \tilde{g}} \frac{\partial \tilde{H}}{\partial y_i}) dy_2\wedge dy_1, \quad \tilde{g}_{ij}= \frac{\partial \tilde{H} }{\partial y_i}\cdot \frac{\partial \tilde{H} }{\partial y_j},\quad \tilde{h}= 2[d\vartheta]\cdot \tilde{H},
\]
and denotes its first variation at the map $H$ as $\mathcal{L}_1$. Then the $H^2(K3)$-valued 2-form induced by the variation $H^{(k)}$ is $\mathcal{L}_1(H^{(k)})$. This can be seen by reexamining the derivation of the weighted maximal submanifold equation in section \ref{Slowcollapsing}.

The next subtle effect happens to the volume normalisation condition in the $SU(3)$-structure. To leading order, the increment of $(h^{[k]})^{1/2}(\omega^{[k]})^{3}$ induced by adjusting $h^{(k)}$ is
\[
3\epsilon^{k+2} ([d\vartheta]\cdot H^{(k)}) h^{-1/2}\omega_y^2\wedge \omega_S,
\]
and the increment of $(h^{1/4}\text{Re}\Omega)^{[k]}\wedge (h^{1/4}\text{Im}\Omega)^{[k]}$ induced by $\rho^{(k)}$ is after an instructive exercise
\[
2\epsilon^{k+2}\omega_S \wedge \omega_y^2 \left(g^{ij}\frac{\partial H}{\partial y_i} \cdot \frac{\partial H^{(k)}}{\partial y_j}\right).
\]
To preserve the volume normalisation
\[
(h^{[k]})^{1/2} (\omega^{[k]})^3= \frac{3}{2} (h^{1/4}\text{Re}\Omega)^{[k]}\wedge (h^{1/4}\text{Im}\Omega)^{[k]} (1+O(\epsilon^{k+1})),
\]
we need to adjust $\omega^{(k-1)}$ by the pullback of a 2-form on $S$
\[
\mathcal{L}_2(H^{(k)})=
\omega_S\left( \tilde{g}^{ij}\frac{\partial H}{\partial y_i} \cdot \frac{\partial H^{(k)}}{\partial y_j}- h^{-1} [d\vartheta]\cdot H^{(k)} \right).
\]
This provides the mechanism that will ultimately fix the ambiguity of $\omega^{(k-1)}$. This operator $\mathcal{L}_2$ has the pleasant interpretation of being the first variation at $H$ of the mapping
\[
\tilde{H}\mapsto \tilde{h}^{-1/2}\sqrt{ \det(\tilde{g}_{ij})} dy_1\wedge dy_2, \quad \tilde{g}_{ij}= \frac{\partial \tilde{H} }{\partial y_i}\cdot \frac{\partial \tilde{H} }{\partial y_j},\quad \tilde{h}= 2[d\vartheta]\cdot \tilde{H}.
\]
Combining the above, the response to $H^{(k)}$ of the $H^2(K3)$-valued 2-form $[\gamma_k]$ is
\[
\mathcal{L}_1 (H^{(k)}) + [d\vartheta] \wedge \mathcal{L}_2(H^{(k)}),
\]
which by comparison with section \ref{Variationalweightedmaximal} is identified as $\mathcal{L}(H^{(k)})$, where $\mathcal{L}$ is the variation of the operator $\delta \mathcal{A}_w$, or equivalently the Hessian operator of the weighted area functional. Our initial problem in this section is now rephrased as solving for $H^{(k)}$ so that $\mathcal{L}(H^{(k)})$ cancels the ASD part of the initial error $[\gamma_k]$, which defines a normal vector field to the image of $H:S\to H^2(K3)$. But by Lemma \ref{secondvariationnegative} the Hessian operator $\mathcal{L}$ is negative definite on normal vector fields with zero boundary data, so $H^{(k)}$ can always be solved uniquely.


We now reexamine the total volume normalisation on $\int h^{[k]}(\omega^{[k]})^3$. In the above procedure, the change in the total volume corresponds to the change in the weighted area $\mathcal{A}_w$, which is zero to leading order because $H$ is a critical point. This ensures the total volume normalisation is preserved.

Now that we have solved for $H^{(k)}$, so as to be able to gain the condition $[\gamma_k]=0$, we can tie up a few loose ends. Notice $\omega^{(k-1)}$ is now fixed.  Then we can follow Remark 
\ref{CYmonopolebetter} to modify $\theta^{(k)}$ by a higher order in $\epsilon$ term to achieve
\[
d\vartheta^{[k]}\wedge (\omega^{[k]})^2=O(\epsilon^{k+2}).
\]
Also, while the steps taken in this section damage the gauge fixing condition on $(h^{1/4}\text{Im}\Omega)^{[k]}\wedge (\omega_1dy_1+\omega_2dy_2)
$, we can repeat some parts in section \ref{DeformationSU3II} to restore the gauge condition. Since this only involve exact 2-forms, it has no effect on the class $[\gamma_k]$. Repeating some steps in section \ref{DeformationSU3I} for one order higher in $\epsilon$, we can ensure $\gamma_k$ lies in the filtration with at least one $dy_i$ term modulo $O(\epsilon^2)$.

\subsection{Deformation of the $SU(3)$-structure IV: linear algebraic constraints}

We now start to modify $\omega^{[k]}$ to improve the $SU(3)$-structure condition. Modulo $O(\epsilon^{k+3})$, at this stage
\[
\omega^{[k-1]}\wedge (h^{1/4}\text{Im}\Omega)^{[k]}= \epsilon^{k+2} dy_1\wedge dy_2 \wedge \text{closed 3-form on K3}
\]
Write the closed 3-form as $df_1\wedge \omega_1+df_2\wedge \omega_2+ df_3\wedge \omega_y$. The $df_3$ term can be removed by adjusting $\rho^{(k)}$ by $d(f_3\omega_S)$. Now the freedom to adjust $\omega^{(k)}$ is $d(ady_1+bdy_2)$. Choosing the functions $a, b$ depending linearly on $f_1, f_2$, we can also remove the $df_1, df_2$ terms. After this adjustment,
\[
\omega^{[k]}\wedge \text{Im}\Omega^{[k]}=O(\epsilon^{k+3}).
\]
The $H^2(K3)$-valued 2-form $[\gamma_k]$ is not affected in this step.

Recall the integrability condition (\ref{integrabilityCYmonopole}) gives $\omega_1\wedge \gamma_k=\omega_2\wedge \gamma_k=0$. There is one extra piece of integrability we have not used: recall (\cf Remark \ref{CYmonopolebetter})
\[
d\vartheta^{[k]}\wedge (\omega^{[k]})^2=O(\epsilon^{k+2}), \quad d\omega^{[k]}=0, 
\]
hence
\[
\gamma_k\wedge \omega_y= \gamma_k\wedge \epsilon^{-1}\omega^{[k]}= \epsilon^{-k-2}\{ \epsilon d\vartheta^{[k]}\wedge (\omega^{[k]})^2+ d((h^{[k]})^{3/4}\text{Re}\Omega^{[k]}\wedge \omega^{[k]})
\}= O(\epsilon)
\]
Namely up to $O(\epsilon)$ error $\gamma_k$ defines a family of ASD harmonic forms on the K3 fibres. But the ASD component of $[\gamma_k]$ has been removed in section \ref{maximalsubmanifoldsystem}, so $\gamma_k=0$, and we have completed the induction.

\section{Global aspects}

We now discuss the phenomena that arise when trying to compactify the iterated fibration structure. These involve both topological and metric aspects. In particular, since any Killing vector field on a compact Ricci flat manifold must be parallel, we need to display mechanisms to break the circle symmetry in order to make nontrivial compact examples possible. The emphasis here is on the geometric picture, without the more serious analytical work.

\subsection{Lefschetz fibration}

Lefschetz fibration is a phenomenon which has relatively little to do with the small circle in $M$, but instead happens to the 6-manifold $M/S^1$. The motivation is that in a global setting, excess $(-2)$-classes are quite inevitable.

Consider first the simpler problem: let $X$ be a Calabi-Yau 3-fold with a holomophic K3 fibration $\pi: X\to\mathbb{P}^1$, where all critical points of $\pi$ have the local complex analytic description as $\pi=z_1^2+z_2^2+z_3^2$ in local coordinates, known as Lefschetz fibrations. Given a K\"ahler class $[\omega_X]$ on $X$, and a K\"ahler class $[\omega_{\mathbb{P}^1}]$ on $\mathbb{P}^1$, the problem is to describe the limiting behaviour of the Calabi-Yau metric $\omega_{CY,\epsilon}$ in the class $[\epsilon \omega_X+\omega_{\mathbb{P}^1}]$ as $\epsilon\to 0$. This can be viewed as a special case, when $M=X\times S^1$ decouples the small circle factor.

This problem has been satisfactorily analyzed in \cite{LiCY}\cite{LiCY2}\cite{LiCY3}.
There are two general lessons:
\begin{itemize}
\item (Multiscaledness) The metrics exhibit different characteristic behaviours at different length scales.

\item
(Universality) The microscopic behaviour of the metric depends essentially only on the local geometry, not on the global information of the Calabi-Yau 3-fold or the Lefschetz fibration.
\end{itemize}

 Globally, the Calabi-Yau metrics converge in Gromov-Hausdorff sense to a singular metric on $\mathbb{P}^1$ which can be explicitly written down via period integrals. The characteristic length scale of the base is $O(1)$. Around  any fixed smooth $K3$ fibre,  the rescaled Calabi-Yau metrics $\epsilon^{-1}\omega_{CY,\epsilon}$ converge smoothly to $K3\times \C$ with the product metric, where the K3 factor carries the unique Calabi-Yau metric in class $[\omega_X]$. Around a singular K3 fibre (which has $\C^2/\Z_2$ orbifold singularity by the Lefschetz condition), the convergence to $K3\times \C$ happens only in a Gromov-Hausdorff sense; here the K3 factor carries the orbifold Calabi-Yau metric. The convergence cannot be smooth because $K3\times \C$ has $\C^2/\Z_2\times \C$ local singularity. The characteristic length scale of the K3 fibre is $O(\epsilon^{1/2})$.

For our current purpose the most interesting aspect is what happens around the critical points of the fibration. The answer involves an exotic complete Calabi-Yau metric $\omega_{\C^3}$ on $\C^3$, whose tangent cone at infinity is $\C^2/\Z_2\times \C$ with the flat orbifold metric. In particular the volume of geodesic balls in $\omega_{\C^3}$ is of order $\sim r^6$, known as maximal volume growth, but $\omega_{\C^3}$ is not the flat Euclidean metric. This exotic metric is constructed by writing down an approximate metric ansatz near the spatial infinity, and using a noncompact version of the Calabi conjecture to perturb it into a Calabi-Yau metric. It is proved in \cite{LiCY3} that after suitable rescaling by a magnification factor  $O(\epsilon^{-4/3})$, the collapsing Calabi-Yau metrics $\epsilon^{-4/3}\omega_{CY,\epsilon}$ around the critical point converge to $\omega_{\C^3}$ as  $\epsilon\to 0$. In particular, the characteristic length scale of the bubble around the critical point is $O(\epsilon^{2/3})$, much smaller compared to the characteristic length scale $O(\epsilon^{1/2})$ of the K3 fibres. The local metric behaviour is captured by the canonical object $\omega_{\C^3}$, independent of the source of the global Calabi-Yau 3-fold.

Back to the 7-dimensional problem where $M/S^1$ around a particular K3 fibre looks approximately like a Lefschetz K3 fibration, we discuss the heuristic reason why the presence of the extra small cicle does not significantly affect the geometry near the critical point, as long as $[d\vartheta]$ is orthogonal to the $(-2)$-class of the vanishing sphere. The effect of this extra $S^1$ comes from its curvature $d\vartheta$. Around the singular K3 fibre, from our ansatz construction in section \ref{Slowcollapsing}, the leading order behaviour of $d\vartheta$ is the harmonic 2-form on the singular K3 in the appropriate class. Since $[d\vartheta]$ is orthogonal to the $(-2)$-class, there is no cohomological incentive for the curvature to concentrate significantly near the critical point. The characteristic curvature scale is $O(1)$. Now when we apply our magnifying glass to see the geometry of the exotic $\C^3$, the curvature effect of $d\vartheta$ is washed away by the scaling factor, \ie microscopically the circle bundle behaves as if it is flat. A further piece of evidence is that $h$ is to leading order constant on fibres. Thus we expect the local geometry is approximately a product of the exotic metric on $\C^3$ with the small $S^1$ factor. In particular, this does not provide a mechanism to break the $S^1$-symmetry.



The monodromy of the Lefschetz fibration is controlled by the Picard-Lefschetz formula. This will require a slight modification to our local discussion of polarized positive sections; instead of taking value in $H^2(K3)$, the section $H$ really takes value in an affine bundle with fibres isomorphic to $H^2(K3)$ (\cf section \ref{Globaltopology} below).

\subsection{Wall crossing}\label{wallcrossing}

\textbf{Wall crossing} refers here to the following phenomenon. Let $\sigma\in H^2(K3)$ be a fixed class orthogonal to $[\omega_y]$. We propose a mechanism for $S$ to be locally divided into two components by a curve $l\subset S$ satisfying $\sigma\cdot H=\text{const}$, such that the first Chern class $\frac{1}{2\pi}[d\vartheta]$
jumps by $\sigma$ when we cross $l$. The mechanism involves allowing the $S^1$-bundle to degenerate along a 3-dimensional discriminant locus fibred over $l$.

\begin{rmk}
Curiously, the condition $\sigma\cdot H=\text{const}$ encodes special Lagrangians in the fast circle collapsing setting (\cf section \ref{fastcirclecollapsing}) where $M/S^1$ is approximately Calabi-Yau. In the fine tuned circle collapsing setting we are studying $M/S^1$ is not close to being Calabi-Yau, so such an interpretation is not available, but the intuition is still useful. 
\end{rmk}

Let $\tau$ be a coordinate along $l$. Our interest is in the behaviour of the $G_2$-structure transverse to $l$. Since the characteristic length scale of the transition (which turns out to be $O(\epsilon)$) is much smaller than the characteristic scale of $l$ which is $O(1)$, it is reasonable to suppose that the geometry is locally approximately independent of $\tau$. A natural suggestion is that the transverse behaviour is modelled on the dimensionally reduced case of Calabi-Yau 3-folds with $S^1$-symmetry, sketched in section \ref{DimreductiontoCY}. What we will now describe is essentially the same as \cite[chapter 4]{SunZhang}, in a slightly different normalisation convention, using the rescaled notation in section \ref{DimreductiontoCY}.

Since $\sigma\cdot dH=0$ along $l$ and $\sigma\cdot [\omega_y]=0$, and noticing how the conversion formula (\ref{conversionformularescaled}) mixes up the hyperk\"ahler structure on K3 surface $D$, along $l$ we have
\[
[\text{Im}\Omega_D]\cdot \sigma=0, \quad [\text{Re}\Omega_D]\cdot \sigma=0.
\]
We assume we can consistently find smooth holomorphic curves $C$ in the class $\sigma\in H^2(K3)$ with respect to the complex structure defined by $\Omega_D$; this can be sometimes guaranteed by cohomological criterions, for example in the case of $(-2)$-classes. From the 7-dimensional perspective, the discriminant locus of the $S^1$-fibration will be the union of these holomorphic curves along $l$. From the dimensional reduction viewpoint, the $\tau$ variable will be dropped, and we think of a Calabi-Yau 3-fold with $S^1$-action, whose fixed points lie over $C$ inside $D\times \{0\}\subset D\times \R_\mu$.

The Calabi-Yau requirement leads to
\[
\begin{cases}
d\vartheta= -\partial_\mu \tilde{\omega}- \epsilon^{-1}d_D^c h \wedge d\mu,
\\
(\partial^2_\mu \tilde{\omega}+ \epsilon^{-1} d_Dd_D^c h)\wedge d\mu= -d(d\vartheta)= -2\pi \delta_C
\\
\tilde{\omega}^2= \frac{1}{2} h\Omega_D\wedge \overline{\Omega}_D.
\end{cases}
\]
This is the same as in section \ref{DimreductiontoCY} except for the codimension 3 distributional current $\delta_C$ along $C\subset D\times \R_\mu$, which arises from the requirement that the integral of $d\vartheta$ on a 2-sphere linking $C$ in $D\times \R_\mu$ should have total integral $2\pi$ to maintain the smoothness of the Calabi-Yau 3-fold. As $\mu$ increases to cross $0\in \R$, the first Chern class $\frac{1}{2\pi}[d\vartheta]$ increases by $\sigma\in H^2(K3)$.

The system so far is highly nonlinear. Fortunately, since the $S^1$-fibres are near the collapsing limit, perturbation techniques are available. We impose the ansatz
\[
\begin{cases}
h=h_0+ \epsilon \mathfrak{h},\quad h_0=const,
\\
\tilde{\omega}=\tilde{\omega}_D+ \epsilon \psi, \quad \tilde{\omega}_D^2= \frac{h_0}{2}\Omega_D\wedge \overline{\Omega}_D.
\end{cases}
\]
This means unless we are very close to the discriminant locus, then $h$ should be locally approximately constant, and $\tilde{\omega}$ should approximately agree with the Calabi-Yau metric on $D$, in line with the basic assumptions in section \ref{Slowcollapsing}. This motivates the linearization of the above system:
\[
\begin{cases}
\tilde{\omega}_D\wedge \psi= \frac{1}{4} \mathfrak{h} \Omega_D\wedge \overline{\Omega}_D, \quad \ie \quad \mathfrak{h}=  h_0 \Tr_{\tilde{\omega}_D} \psi,
\\
\epsilon \partial_\mu^2 \psi+  d_Dd_D^c \mathfrak{h}= -2\pi \delta_C,
\end{cases}
\]
hence the Laplace equation
\[
(\epsilon \partial_\mu^2 -  h_0 d_Dd_D^* )\psi= -2\pi \delta_C.
\]
The leading singularity along $C$ is
\[
\begin{cases}
\psi\sim \frac{\sqrt{-1}}{4r} \eta\wedge \bar{\eta},
\\
\mathfrak{h}\sim \frac{1}{2r},
\\
r\sim \sqrt{  \epsilon h_0^{-1}\text{dist}_{\tilde{\omega}_D}(\cdot, C)^2+ |\mu|^2 },
\end{cases}
\]
where $\eta$ denotes a $(1,0)$-form on $D$ normal to $C$, with $\Tr_{\tilde{\omega}_D} \frac{h_0\sqrt{-1}}{2}\eta\wedge \bar{\eta}=1$.

To give the geometric interpretation, first observe that the correction effect is significant only when $r\lesssim \epsilon$, corresponding to length scale $O(\epsilon)$, which is comparable to the average length of $S^1$-fibres and much smaller compared to the length scale $O(\epsilon^{1/2})$ of the K3 surfaces.

The correction is dominant in the small neighbourhood around $C$, and its effect is to make the geometry transverse to $C$ look approximately like Taub-NUT metrics. In fact, we have a metric asymptote near $C$
\[
g'\sim \epsilon \tilde{g}_D + \frac{\epsilon^2}{2r}\eta\otimes \eta+ (h_0+ \frac{\epsilon}{2r}) (d\mu)^2+ (h_0+ \frac{\epsilon}{2r}) \epsilon^2\vartheta^2,
\]
and the component of the K3 metric $\tilde{g}_D$ normal to $C$ is $h_0\eta\otimes \eta$ by construction. We recognize the metric component transverse to $C$
\[
(h_0+ \frac{\epsilon}{2r}) \{(d\mu)^2+ \epsilon  \eta\otimes \eta \}+ (h_0+ \frac{\epsilon}{2r}) \epsilon^2\vartheta^2
\]
is up to scaling just Taub-NUT metric in disguise. In particular, while one may a priori worry about the validity of linear approximations when the correction is large, it turns out that after adding the linear correction, the structure we obtain is close to being Calabi-Yau. This ansatz is used in \cite{SunZhang} to describe the neck region for certain degenerations of Calabi-Yau metrics. Our proposal here is that the same picture occurs for iterated collapsing 
of $G_2$-manifolds.

In this mechanism, the $S^1$-bundle becomes degenerate over some discriminant locus, but the global $S^1$-symmetry is still not broken; the singularity just corresponds to the fixed points of the $S^1$-action.

\begin{rmk}
Under additional hypotheses, we can incorporate wall crossing phenomenon with two transversely intersecting walls $l_1, l_2$. The intersection  $l_1\cap l_2$ corresponds to a K3 fibre containing holomorphic curves $C_1,C_2$ in the class $\sigma_1$, $\sigma_2$ respectively. One needs to be careful that the two complex structures on the K3 surface corresponding to $l_1, l_2$ are different. We now assume that $C_1$ and $C_2$ do not intersect; this is in general stronger than the homological condition $[\sigma_1]\cdot [\sigma_2]=0$. We can then expect that the wall crossing across $l_1, l_2$ behave independently. The reason is that the characteristic length scale of the wall crossing in $O(\epsilon)$, which is much less than the characteristic scale of the K3 fibre $O(\epsilon^{1/2})$, so the linear correction induced by the distributional terms supported on $C_1, C_2$ have neglegible influence on each other.

If $C_1$ and $C_2$ are intersecting, then the local situation is supposedly modelled on a collapsing $G_2$-manifold with $S^1$-fibration over Euclidean $\C^3$  whose discriminant locus is the union of two intersecting special Lagrangian 3-planes. This conjectural metric model is a folklore prediction of Atiyah and Witten \cite{AtiyahWitten} and its unsolved existence problem is a major difficulty in the Foscolo-Haskins-Nordstr\"om picture \cite{HaskinsS1}.
\end{rmk}

\subsection{ALF fibration}\label{ALFfibration}

Instead of the Taub-NUT metric, we can modify the distributional equation to
\[
(\partial^2_\mu \tilde{\omega}+ \epsilon^{-1} d_Dd_D^c h)\wedge d\mu= -d(d\vartheta)= -2\pi (k+1) \delta_C, 
\]
so the solution is modified into
\[
\begin{cases}
\psi\sim \frac{\sqrt{-1}(k+1)}{4r} \eta\wedge \bar{\eta},
\\
\mathfrak{h}\sim \frac{k+1}{2r},
\\
r\sim \sqrt{  \epsilon h_0^{-1}\text{dist}_{\tilde{\omega}_D}(\cdot, C)^2+ |\mu|^2 },
\end{cases}
\]
The transverse structure to $C$ matches with the asymptote of the $A_k$ type ALF gravitational instantons. Since ALF gravitational instantons in general have nontrivial moduli, one expects internal degrees of freedom supported on the discriminant locus to arise when we attempt to glue in an ALF fibration, similar to the harmonic 1-forms appearing in Joyce and Karigiannis \cite{JoyceKarigiannis}. 
 We will be brief on issues already discussed extensively for ALE fibrations in the physics literature \cite{Pantev}\cite{Barbosa}\cite{Hubner1}\cite{Hubner}.

The basic situation is that we are given a 3-manifold with a given Riemannian metric, and we wish to construct an approximate $G_2$-metric on a fibration over the 3-manifold with fibres diffeomorphic to $A_k$-type ALF space, whose asymptotic behaviour near the fibrewise infinity is prescribed, and the metric restricted to each fibre is approximately isomorphic to some choice of $A_k$-type ALF space with characteristic length scale $O(\epsilon)$ much smaller compared to the length scale of the 3-manifold. Morever, for simplicity we assume that the asymptotic circle length of the ALF fibres is approximately constant over the region of the 3-manifold under consideration, so that the variation of the ALF structures can be captured completely by period integrals.
In our case of interest the 3-manifold is the  discriminant locus fibred over $l$ with fibre $C$ of length scale $O(\epsilon^{1/2})$, and the asymptotic circle length is indeed almost constant within length $O(\epsilon^{1/2})$, although not necessarily so over length scale $O(1)$.

Taking an orthonormal frame $e_1, e_2, e_3$ on the 3-manifold, we can write the ansatz for the $G_2$-structure as
\[
\phi\sim \epsilon^2(\omega_1^{ALF} \wedge e_1^*+ \omega_2^{ALF} \wedge e_2^*+\omega_3^{ALF} \wedge e_3^*)-e_1^*\wedge e_2^*\wedge e_3^*,
\]
where $\omega_i^{ALF}$ are fibrewise hyperk\"ahler triples of $A_k$-type ALF metrics with prescribed asymptote at fibre infinity. (Strictly speaking, to make sense of $\omega_i^{ALF}$ on the total space, we need the data of a horizontal distribution; its determination is similar to  Donaldson's proposal of adiabatic K3 fibration \cite{Donaldson}).

Now working in geodesic coordinates around a given point, in the adiabatic limit $\epsilon\to 0$, the requirement for $d\phi=0$ leads to that the $H^2(ALF)$-valued 1-form $ \lambda=\sum_i[\omega_i]e_i^*$ is closed, and the requirement for $d*_\phi \phi=0$ leads to $d*\lambda=0$ on the 3-manifold. To summarize, as long as the adiabatic approximation holds and the fibre asymptotic behaviour stays approximately constant, then the fibration is governed by an $H^2(ALF)$-valued harmonic 1-form locally on the 3-manifold.


Thinking more globally on the 3-manifold, $H^2(ALF)$ becomes a local system, alternatively viewed as a bundle with a flat connection, and $\lambda$ is a harmonic 1-form valued in this bundle. Physicists view such data as an Abelian solution to the BPS equation \cite[section 2.3]{Hubner}, also related to spectral covers \cite[section 2.4]{Pantev}.  In our case of interest, the 3-manifold is itself fibred over $l$ with very small fibres $C$, and to leading order we get bundle valued harmonic 1-forms over the Riemann surfaces $C$.

Some sources of difficulty are:
\begin{itemize}
\item
In elementary examples of the 3-manifolds, typically there are no nontrivial harmonic 1-forms. It has been suggested that we should relax the harmonic 1-form condition along some collection of knots inside the 3-manifold, and build in nontrivial monodromy for the local system around these knots. The geometric situation seems to correspond to the Kovalev-Lefschetz fibrations discussed by Donaldson \cite{Donaldson}, adapted to the ALF setting. It is also suggested that we may consider branched covers over the 3-manifold, so the harmonic 1-forms may be multivalued \cite[section 9]{Barbosa}. Geometrically, the small ALF instantons are arranged in several sheets.

\item
On a generic 3-manifold, it is commonly expected that when the $H^2(ALF)$-valued 1-form evaluated on a (-2)-curve has a zero, the $G_2$-manifold should develop an isolated singularity modelled on the Bryant-Salamon cone over $\mathbb{CP}^3$, producing chiral matter in the physics literature, but the metric local model seems difficult to construct rigorously.

The fact that our 3-manifold is itself highly collapsed somewhat alters this picture. Thinking in terms of the dimensional reduction to Calabi-Yau 3-fold, we suggest the following picture may be relevant for the simplest case, and is accessible within current technology. Complex geometrically, consider the quadric cone
\[
\{ z_1^2+z_2^2+z_3^2+z_4^2=0    \}\subset \C^4,
\]
admitting a fibration 
\[
\{ z_1^2+z_2^2+z_3^2+z_4^2=0    \} \xrightarrow{z_4} \C,
\]
with fibres being affine quadric surfaces. These agree with the underlying space of $A_1$-type ALF gravitational instantons, so we can hope to construct an ansatz K\"ahler metric asymptotic to an ALF fibration near infinity. One expects that by a noncompact version of the Calabi conjecture, the ansatz metric can be corrected into a Calabi-Yau metric. The result of Hein and Sun \cite{HeinSun} then suggests that the Calabi-Yau metric is asymptotic to the Stenzel cone at the origin.  Thus this Calabi-Yau metric will furnish the transition between the asymptotic ALF fibration picture and the local Stenzel cone.

\item
We also wish to know the long range behaviour over $l$. One expects that the variation of the asymptotic circle length, and the moduli of Abelian solutions of the BPS equation to play a role. Morever, assuming the above Calabi-Yau model metric can be constructed, then na\"ively we have a codimension 6 singularity along $l$ transversely modelled on the Stenzel cone. Gao Chen \cite{Chen} has studied the deformation theory of such singularities, and found an infinite dimensional obstruction. The  expectation is that the codimension 6 singularity breaks up into isolated codimension 7 singularities modelled on the cone over $S^3\times S^3$.

\end{itemize}

Of course, these rather formidable difficulties only arise if $k\geq 1$. The $k=0$ case with the Taub-NUT metric does not have internal moduli, so does not exhibit these difficult yet rich phenomena.

\subsection{Orientifold and boundary behaviour}\label{Orientifoldboundary}

We now discuss an analogue of the wall crossing phenomenon in section \ref{wallcrossing} involving orbifold behaviour, which will give rise to boundaries in $S$.

Since the wall crossing phenomenon has characteristic length scale $O(\epsilon)$, the variation of the hyperk\"ahler structure on K3 fibres  is negligible. We will assume the existence of a $\Z_2$-symmetry acting as a \textbf{non-symplectic involution} on the K3 surface
\[
\rho^* \tilde{\omega}_D=\tilde{\omega}_D, \quad \rho^*\Omega_D=-\Omega_D.
\]
There is a classification for K3 surfaces with non-symplectic involutions in terms of K3 lattice theory (\cf \cite[section 3]{KovalevLee} for a short survey).
In particular, the fixed locus of $\rho$
is a disjoint union of smooth holomorphic curves $C_i$ in $D$. This $\rho$ extends to an action
\[
\rho^* \tilde{\omega}_D=\tilde{\omega}_D, \quad \rho^*h=h,\quad \rho^*\Omega_D=-\Omega_D, \quad \rho^*\vartheta=-\vartheta, \quad \rho^*\tau=\tau, \quad \rho^*\mu=-\mu.
\]
In terms of the Apostolov-Salamon $SU(3)$-structure,
\[
\rho^*\omega=-\omega, \quad \rho^*\Omega=\overline{\Omega}.
\]
This behaves like an antiholomorphic involution.
Similar to section \ref{ALFfibration}, we modify the distributional equation of section \ref{wallcrossing} to
\[
(\partial^2_\mu \tilde{\omega}+ \epsilon^{-1} d_Dd_D^c h)\wedge d\mu= -d(d\vartheta)= -2\pi \sum_i(2m_i-4) \delta_{C_i}, 
\]
so the linear correction term is modified into
\[
\begin{cases}
\psi\sim \frac{\sqrt{-1}(2m_i-4)}{4r} \eta\wedge \bar{\eta},
\\
\mathfrak{h}\sim \frac{2m_i-4}{2r},
\\
r\sim \sqrt{  \epsilon h_0^{-1}\text{dist}_{\tilde{\omega}_D}(\cdot, C_i)^2+ |\mu|^2 }.
\end{cases}
\]

The $\Z_2$ symmetry constrains $[d\vartheta]$. To see this, notice wall crossing predicts that as $\mu$ increases across zero, $\frac{1}{2\pi}[d\vartheta]$ increases by $\sum_i(2m_i-4)[C_i]\in H^2(K3)$. Compatibility with $\rho^*\vartheta=-\vartheta$ then implies
\begin{equation}\label{firstChernclassnearboundary}
\frac{1}{2\pi}[d\vartheta]=
\begin{cases}
\sum (2-m_i)[C_i], \quad \mu<0,
\\
\sum (m_i-2)[C_i],\quad \mu>0.
\end{cases}
\end{equation}
Morever, we have the cohomological constraints
\begin{equation}\label{boundarycondition}
\rho^*[\omega_y]=-[\omega_y],\quad \rho^*\frac{\partial H}{\partial \tau}=- \frac{\partial H}{\partial \tau}, \quad \rho^*\frac{\partial H}{\partial \mu}= \frac{\partial H}{\partial \mu},
\end{equation}
where $\frac{\partial}{\partial \tau}$ is tangent to $l$ and $\frac{\partial}{\partial \mu}$ is normal to $l$ with respect to the induced metric $g$ on $S$.

 We now take the $\Z_2$ quotient of the construction. The structure transverse to $C_i$ now matches precisely with the asymptote at infinity of $D_{m_i}$ type ALF gravitational instantons. This suggests that we should glue in a \textbf{family of $D_{m_i}$-type ALF gravitational instantons} fibred over $C_i$ in the Calabi-Yau case (or the 3-manifold swept out by $C_i$ over $l$ in the $G_2$-case). For $m=0$, the $D_0$-type ALF gravitational instanton is the Atiyah-Hitchin manifold, which has no deformation except for scaling, so we expect that the matching condition essentially determines this ALF fibration up to small error. For $m>0$, the $D_m$ instantons have internal moduli, so we expect extra data to be necessary in specifying the ALF fibration, in the same spirit as in section \ref{ALFfibration}.

\begin{rmk}
We saw in section \ref{ALFfibration} that zeros of harmonic 1-forms typically lead to difficult singularity problems. There is one case where we can hope the $D$-type is better behaved: the $D_1$-type corresponds to deformations of the double cover of the Atiyah-Hitchin manifold. In particular, the zero harmonic 1-form corresponds to taking a fibration by the the double cover of Atiyah-Hitchin manifold over the 3-dimensional discriminant locus, without creating any singularity. 	
\end{rmk}

An important conceptual point is that unlike the $A_k$-type gravitational instantons, the $D_m$-type can \textbf{break the global $S^1$-symmetry}, as must be the case for nontrivial compact $G_2$-manifolds.

Another conceptual feature is the effect of $\Z_2$-quotient. While the quotient results in codimension 4 orbifold singularity, this is resolved when we glue in the ALF fibration, so $M$ is still smooth. (In some cases such as $m_i=2$, we can also choose to not glue in the ALF fibration, which will result in codimension 4 orbifold singularity on $M$). Nevertheless, the $\Z_2$-quotient identifies $\mu$ with $-\mu$, so instead of two components of $S$ separated by $l$, there is only one component with boundary $l$. Even though we ultimately want to compactify $M$ into a closed manifold, we should allow $S$ to be a \textbf{manifold with boundary} instead.

\begin{rmk}
It may be asked what happens when two boundary curves $l$ intersect. One possible guess is to introduce corners to the Riemann surface with boundary, but the $G_2$-geometry seems unclear in general.
\end{rmk}

\subsection{Tian-Yau region}\label{TianYau}

There is yet another mechanism to break the global $S^1$-symmetry, again coming from the dimensional reduction to Calabi-Yau 3-folds. This plays a prominent role in \cite{SunZhang} to describe the end regions of certain degenerating Calabi-Yau metrics. More backgrounds can be found in \cite[section 3]{HSVZ}.

Let $Y$ be an $n$-dimensional Fano manifold, $D$ a smooth anticanonical divisor in $Y$, and denote $Z=Y\setminus D$. By adjuntion $D$ is itself a Calabi-Yau manifold. Fix a defining section $S$ of $D$, and view $S^{-1}$ as a holomorphic $n$-form $\Omega_Z$ on $Z$ with a simple pole along $D$, whose Poincar\'e residue defines a holomorphic volume form $\Omega_D$ on $D$. Under suitable normalisation
\[
\frac{n!}{2^n} \int_D \sqrt{-1}^{(n-1)^2} \Omega_D\wedge \overline{\Omega}_D= (2\pi c_1(K_Y^{-1}|_D)) ^{n-1}.
\]
By Yau's theorem, there is a unique Calabi-Yau metric $\omega_D$ in the class $2\pi c_1(K_Y^{-1}|_D)$. The form $-\sqrt{-1}\omega_D$ is the curvature form of a Hermitian metric $\norm{\cdot}$ on $K_Y^{-1}|_D$; we fix a smooth extension of the Hermitian metric to $Z$, preserving the positivity of the curvature.

The form
\[
\omega_Z= \frac{n}{n+1} \sqrt{-1}\partial \bar{\partial} (-\log \norm{S}^2)^{\frac{n+1}{n} }
\]
is K\"ahler on a neighbourhood of infinity of $Z$. By a noncompact version of the Calabi conjecture, one finds the \textbf{Tian-Yau metric} $\omega_{TY}=\omega_Z+\sqrt{-1}\partial \bar{\partial }\phi$ on $Z$ asymptotic to $\omega_Z$ at infinity, satisfying the Calabi-Yau condition
\[
\omega_{TY}^n= \frac{n!}{2^n} \sqrt{-1}^{n^2} \Omega_Z\wedge \overline{\Omega}_Z.
\]

The asymptotic geometry is identified as the \textbf{Calabi ansatz}. Complex geometrically, the neighbourhood of $D$ inside $Y$ is well approximated by the normal bundle of $D$, namely the total space of $K_Y^{-1}|_D\to D$. The holomorphic volume form $\Omega_Z$ is approximately $d\log \xi\wedge \Omega_D$, where $\xi$ denotes a local holomorphic variable on the fibres of $K_Y^{-1}|_D\to D$. The asymptotic metric $\omega_Z$ is approximated by
\[
\omega_{Calabi}=
 \frac{n}{n+1} \sqrt{-1}\partial \bar{\partial} (-\log \norm{\xi}^2)^{\frac{n+1}{n} }
\]
using the Hermitian metric on $K_Y^{-1}|_D$. This setup is invariant under the $S^1$-action on the fibres of the line bundle. The moment map is
\[
\mu= ( -\log \norm{\xi}^2  )^{1/n},
\]
and the Calabi ansatz can be written in the framework of section \ref{DimreductiontoCY} as 
\[
\begin{cases}
\omega_{Calabi}= \vartheta\wedge d\mu+ \mu \omega_D,
\\
\Omega_{Calabi}= d\log \xi\wedge \Omega_D= - (  \frac{n}{2}(-\log \norm{\xi})^{(n-1)/n}  d\mu- \sqrt{-1}\vartheta) \wedge\Omega_D,
\\
d\vartheta=-\omega_D,
\\
h_{Calabi}= \frac{n}{2}(-\log \norm{\xi})^{(n-1)/n }= \frac{n}{2}\mu^{n-1}.
\end{cases}
\]

For our purpose the case of interest is $n=3$ and $D=K3$. Then the above description agrees with Example \ref{Calabiansatz}, up to suitable scaling factors. The significance is that at least in the dimensionally reduced setting, up to rescaling the asymptote of the Tian-Yau metric matches the adiabatic ansatz with 
\[
H= [\text{Re}\Omega_D]\tau+ \frac{1}{2}[d\vartheta] \mu^2
\]
(compare the end of section \ref{DimreductiontoCY}), so that by gluing in the Tian-Yau metric we can desingularize the $\mu=0$ boundary locus of the adiabatic ansatz where $h=0$. The characteristic length scale of the Tian-Yau bubble region is $O(\epsilon^{2/3})$. Thus we can at least hope that the Tian-Yau regions contribute to another type of boundary for $S$. More evidence is needed to test this gluing proposal, since the weighted maximal submanifold equation may have poor regularity near the $h=0$ boundary, caused by the degeneracy of ellipticity.

\subsection{Global topology}\label{Globaltopology}

We now discuss some topology in order to set up a global version of the weighted maximal submanifold equation. The base $S$ will be a Riemann surface with boundary, with isolated interior points $S_{sing}$ over which the K3 fibres are supposed to develop Lefschetz singularities. The complement of these points in $S$ is denoted $S_{sm}$, where the K3 fibres will be smooth. The interior of $S$ can also contain a number of walls, which are unions of disjoint circles, not touching the boundary or $S_{sing}$. (These requirements are meant to be tentative, presumably relaxable if we have better understanding of local metric models.)
The precise location of the walls should not be fixed a priori, and behave instead as in a free boundary problem. The topological setup will come directly from abstracting the cohomological/lattice theoretic aspects of the Lefschetz fibration, wall crossing, and the orientifold boundary behaviour. It is  conceivable that other local mechanisms could be incorporated into this formalism.

First we consider the smooth K3 fibration $(M/S^1)_{sm}\to S_{sm}\subset S$, and seek cohomological descriptions of the adiabatic solution to the Apostolov-Salamon equation. The fibration induces an $H^2(K3,\Z)$ local system $\Gamma$ over the smooth locus $S_{sm}$. The symplectic form $\omega$ defines a constant section $[\omega_y]$ of $\Gamma_\R=\Gamma\otimes_\Z \R$.

On a local chart $U_\alpha$, 
the information of $h^{1/4}\text{Im}\Omega$ is encoded by a polarized positive section $H_\alpha: U_\alpha\to H^2(K3)$, determined up to constant. On the overlaps of two charts $U_\alpha$ and $U_\beta$, the difference $H_\alpha-H_\beta$ is locally constant. The collection
$
\{ H_\alpha- H_\beta   \}
$
defines a Cech cocycle in $H^1(S_{sm}, H^2(K3))$. More invariantly, the Leray spectral sequence gives
\[
H^3((M/S^1)_{sm} ) \simeq H^1( S_{sm}, \Gamma_\R )
\]
and the class of $h^{1/4}\text{Im}\Omega$ is mapped to a class in $H^1(S_{sm}, \Gamma_ \R)$. Via the cocycle, we can regard the collection $\{H_\alpha\}$ as defining a section $H$ of an affine bundle $\tilde{\Gamma}_\R$, whose associated vector bundle is $\Gamma_\R$.

The class $\frac{1}{2\pi}[d\vartheta]$ is viewed as a discontinuous section of $\Gamma$ with precise jumping behaviours across the walls.
The boundary value of $[d\vartheta]$ is prescribed in (\ref{firstChernclassnearboundary}). More precisely, we equip each boundary circle of $S$ with an involution $\rho^*$ acting on the fibres of $\Gamma|_{S^1}$, which is the cohomological version of non-symplectic involutions. This $\rho^*$ determines the class of the components $C_i$ of the fixed point locus of the non-symplectic involution. With additional choice of integers $m_i$ we prescribe $[d\vartheta]$ according to the formula (\ref{firstChernclassnearboundary}). On the complement of the walls, the class $[d\vartheta]$ is locally constant. Each (oriented) wall is associated with a class $\sigma\in \Gamma$ orthogonal to $[\omega_y]$, and $\frac{1}{2\pi}[d\vartheta]$ jumps by $\sigma$ across the wall.

In each local chart $H_\alpha$ is a polarized positive section, namely 
\[
H_\alpha\cdot [\omega_y]=0, \quad h=2H_\alpha\cdot [d\vartheta]+ \text{const}(\alpha)>0, \quad  \tilde{g}= \frac{\partial H_\alpha}{\partial y_i} \cdot \frac{\partial H_\alpha}{\partial y_j} dy_idy_j \text{ Riemannian}.
\]
The constants here arise from the affine ambiguity of $H_\alpha$.
The gradient of $H_\alpha$ should be continuous across the wall, and $h$ should be independent of the charts as it is related to the circle length. To be compatible, we must require along the wall
\[
H_\alpha \cdot \sigma=\text{const}(\alpha).
\]
The precise position of the walls are thus tied up with the unknown $H$, and must be solved alongside $H$ rather than prescribed a priori. Furthermore, over $S_{sm}$ there are no singular K3 fibres by assumption, so cohomologically we require there is no excess $(-2)$-class for $H$ over $S_{sm}$.

We now turn to the boundary of $S$. For simplicity we only consider orientifold type boundary described in section \ref{Orientifoldboundary}. The boundary condition is specified by (\ref{boundarycondition}), namely we constrain the tangential derivative of $H$ to the $-1$ eigenspace of $\rho^*$, and the normal derivative of $H$ to the $+1$ eigenspace of $\rho^*$, where the normal direction is with respect to the metric induced from immersion in $H^2(K3)$. In particular, in some preferred local chart $H$ restricted to the boundary lands in the $-1$ eigenspace.
These are a mixture of Dirichlet and Neumann boundary conditions of complementary dimensions.

Next we deal with the singular K3 fibres. We shall assume that the K3 fibration is Lefschetz. To incorporate the Picard-Lefschetz monodromy, we consider $S$ as an orbifold locally modelled on $\C/\Z_2=\C/\iota$ near $p\in S_{sing}$, and
 require that $\tilde{\Gamma}_\R$ is a flat orbifold affine bundle. The local sections of $\tilde{\Gamma}_\R$ in an open ball $U_\alpha=\tilde{U}_\alpha/\iota$ around $p$ are concretely written in some chart as maps $\tilde{H}_\alpha: \tilde{U}_\alpha\to H^2(K3)$ satisfying an equivariance condition
 \[
 \tilde{H}_\alpha(\iota(z) )= \tau_\alpha (\tilde{H}_\alpha(z)), \quad \tau_\alpha(v)=v+(\delta\cdot v)\delta , \quad z\in \tilde{U}_\alpha,
 \]
where $\delta$ is the $(-2)$-class of the vanishing sphere, orthogonal to $[\omega_y]$. The class $[d\vartheta]$ is required to be monodromy invariant. To incorporate Lefschetz fibrations into the notion of polarized positive sections, we require that around $p$, up to the ambiguity of an affine constant,
\[
\tilde{H}_\alpha(z)= v_1 \text{Re}(z^2)+ v_2 \text{Im}(z^2)+  \text{Im}( b z^3 )\delta +O(z^4),
\]
where $v_1, v_2\in H^2(K3)$ satisfy 
\[
v_i\perp [\omega_y], \quad v_i \perp \delta, \quad (v_i\cdot v_j) =\lambda \delta_{ij},\quad \lambda>0,
\]
and $b\in \C$ is nonzero, corresponding to the nondegeneracy of the Lefschetz singularity. The function $h$ is required to be positive and at least $C^1$.

The problem to solve the \textbf{global weighted maximal submanifold equation} means to find the configuration of the walls and boundaries, such that the PDE (\ref{weightedmaximalsubmanifold}) is satisfied on the complement of the walls and boundaries, the mixed boundary condition is satisfied on $\partial S$, the jumping behaviour across the walls and the local behaviour near the Lefschetz singularities are as prescribed, and there are no excess $(-2)$-classes so that no other singular K3 fibres can appear.

The topological consistency of our setup has an interesting consequence closely related to \textbf{charge conservation} (\cf section \ref{fastcirclecollapsing}). Recall that the walls carry $\Gamma$-valued sections $\sigma$ prescribing the jumping of $[d\vartheta]$, so define $\Gamma$-valued 1-clycles. Similarly with the boundary components of $S$, carrying the $\Gamma$-valued 1-cycles coming from (\ref{firstChernclassnearboundary}). We claim the sum of all these $\Gamma$-valued 1-cycles equals zero in homology.
This is because $[d\vartheta]$ is locally constant in each domain $S_i$ bounded by walls and boundaries, and when we take the boundary of $\sum_i [d\vartheta]S_i$, the wall contributions cancel in pairs to yield the desired linear relation.



\subsection{Variational viewpoint}

The local version of the weighted maximal submanifold equation has the interpretation as the critical point condition of a weighted area functional. We now fix all the global topological data, and compute formally the first variation of the weighted area functional
\[
\mathcal{A}_w(H)=  \int_S h^{1/2} \sqrt{\det \tilde{g} }dy_1\wedge dy_2, 
\]
 allowing for the deformation of the walls and boundaries subject to the various boundary conditions. The local discussion in section \ref{Variationalweightedmaximal} already shows that the weighted maximal submanifold (\ref{weightedmaximalsubmanifold})  holds for critical points, so the task here is to examine boundary terms.

In each chamber $S_i$ bounded by some walls and boundaries, given an arbitrary variation $f=\delta H$ which is a local section of $\Gamma_\R$, and assuming (\ref{weightedmaximalsubmanifold}), the contribution to $\delta \mathcal{A}_w$ is the boundary line integral
\[
\int_{\partial S_i} h^{1/2} \frac{\partial H}{\partial \nu}\cdot f dl,
\]
where $\frac{\partial H}{\partial \nu}$ denotes the normal derivative with respect to the induced metric, and $dl$ denotes the induced line element on $\partial S_i$. Notice the affine ambiguity of $H$ is eliminated in taking the derivative. The Lefschetz singularities occur over isolated points, and give no contribution to these integrals.

The boundaries of $S_i$ can be the walls or the boundaries of $S$. Since $f, h$ and the gradient of $H$ are required to be continuous across the walls, the wall contributions cancel in pairs. As for the boundary of $S$, our mixed boundary condition requires the normal derivative of $H$ to be in the $+1$ eigenspace of the cohomological involution $\rho^*$, and in some preferred chart the boundary value of $H$ lands in the $-1$ eigenspace. For the first variation to be compatible with the Dirichlet part of the boundary conditions, the boundary value of $f$ needs to land in the $-1$ eigenspace. But the two eigenspaces of $\rho^*$ are orthogonal, so $f$ is orthogonal to $\frac{\partial H}{\partial \nu}$ on the boundary of $S$, hence
\[
\sum \int_{\partial S_i} h^{1/2} \frac{\partial H}{\partial \nu}\cdot f dl=0.
\]
This shows the boundary conditions are compatible with the critical point interpretation.

\subsection{Open questions}

In the program to ultimately construct new compact examples of $G_2$-manifolds, there are 3 types of questions to be addressed:

\begin{itemize}
\item Find consistent topological data, involving the Riemann surface, the topological configuration of walls, boundaries and singularities, the affine bundle $\tilde{\Gamma}_\R$, the lattice theoretic data on the boundary and the walls, the polarization class $[\omega_y]$, and the first Chern class $[d\vartheta]$ in each chamber. These data should allow for the existence of a global polarized positive section $H$. This problem has a flavour analogous to K3 matching problems involved in twisted connected sum constructions. 

\item
Solve a free boundary type problem to find a solution $H$ to the global weighted maximal submanifold equation, along with the position of the walls. There should be no excess $(-2)$-classes, so the only singular K3 fibres are those nodal fibres of the Lefschetz fibration.
This step deals with the nonlinearity of the Apostolov-Salamon equation, and the variational interpretation is likely useful. One also needs to develop a good regularity theory for the weighted maximal submanifold equation; most of this difficulty comes from the walls and boundaries, where the jumping behaviour of $[d\vartheta]$ indicates that $H$ cannot be smooth.

\item
Perform a gluing construction to construct the $G_2$-metric, taking into account the various geometric ingredients. As discussed, this combines features of many known constructions, but there are still quite substantial challenges. For instance, one needs to develop a good deformation theory for the local $G_2$-metrics on the ALF fibrations. 

\end{itemize}

Some other natural questions include:

\begin{itemize}
\item Are these potential examples deformation equivalent to examples arising from the Foscolo-Haskins-Nordsr\"om style fast circle collapsing picture, by changing the cohomological class of the $G_2$ 3-form?

\begin{rmk}
Most of our geometric ingredients have fast collapsing analogues. For instance,  the wall crossing is analogous to the adiabatic special Lagrangians in section \ref{fastcirclecollapsing}, and the orientifold locus is analogous to the fixed locus of an antiholomorphic involution in the Calabi-Yau 3-fold base.	
\end{rmk}


\item
How can one find new complete non-compact examples? We remark that a large part of our formal picture carries over if we replace K3 surfaces by ALE spaces.
  
\item
How about the special submanifolds and gauge theory on these potential examples?

\item
Can one incorporate more local mechanisms and thereby relax the tentative topological hypotheses (such as the disjointness of the walls and the boundary components, and the no excess $(-2)$-class condition)? As they stand, these topological conditions are rather restrictive.

\item
On the other hand, is there any a priori restriction on the topological complexity no matter how much one relaxes the topological hypotheses? In particular, what is the topological significance of the positivity of the Bakry-\'Emery Ricci curvature?
\end{itemize}






\section{$Spin(7)$	analogues}

We now sketch a very analogous ansatz in the context of $Spin(7)$ geometry. This involves a small circle bundle over a 7-dimensional manifold with a closed $G_2$-structure, admitting a K3 fibration over a 3-fold base. This construction dimensionally reduce to our $G_2$ case, when the 8th dimension splits off. From a different perspective, it is a generalization of Donaldson's adiabatic coassociative K3 fibration proposal, which appears when the small circle bundle is almost flat.

\subsection{Small circle limit of $Spin(7)$ manifolds}

The Foscolo-Haskins-Nordstr\"om picture has an analogue on $Spin(7)$-manifolds. The $Spin(7)$ 4-form $\Phi$ with $S^1$-symmetry can be described in terms of a $G_2$-structure $\bar{\phi}$ on the 7-fold base, together with the $S^1$ connection $\vartheta$, and a positive function $h$ measuring the inverse squared length of the Killing vector field:
\begin{equation}
\Phi=\epsilon \vartheta\wedge \bar{\phi} + \bar{\psi}, \quad \bar{\psi}= h^{2/3} *_{\bar{\phi}} \bar{\phi}, \quad g= \epsilon^2 h^{-1}\vartheta^2+ h^{1/3} g_{\bar{\phi}},
\end{equation}
where $*_{\phi} \bar{\phi}$ and $g_{\bar{\phi}}$ are the 4-form and the metric associated to the $G_2$-structure $\bar{\phi}$ on the 7-fold. The torsion free condition is $d\Phi=0$, or equivalently 
\[
d\bar{\phi}=0, \quad d\bar{\psi}+ \epsilon d\vartheta\wedge \bar{\phi}=0.
\]
Without any fine tuning, the formal limit as $\epsilon\to 0$ is that $\bar{\phi}$ is a torsion free $G_2$-structure, and $h\approx \text{const}$. To see the first order correction, we write $h=1+\epsilon \mathfrak{h}$. Since $d\bar{\phi}=0$, we know $d *_{\bar{\phi}} \bar{\phi}$ lies in the 14-dimensional component of $\Omega^5$, whence we deduce from $d\bar{\psi}+ \epsilon d\vartheta\wedge \bar{\phi}=0$ that to leading order
\[
\pi_7( d\vartheta \wedge \bar{\phi}+ \frac{2}{3} d \mathfrak{h} \wedge *_{\bar{\phi}} \bar{\phi} )=0.
\]
This is recognized as the $G_2$-monopole equation. The analogue of Dirac pole singularity along special Lagrangians inside the Calabi-Yau 3-fold base, is coassociative manifolds inside the $G_2$-manifold base.

\subsection{Iterated fibration: fast circle collapsing}\label{Spin7dimrectiontoG2}

Now we consider the 7-fold with the closed $G_2$-structure $\bar{\phi}$ as admitting a collapsing coassociative K3 fibration in its own right, over a 3-dimensional base $B$ with coordinates $y_0, y_1,y_2$. Writing out in the type decomposition of forms,
\begin{equation}\label{Spin7structurecomponents}
\begin{cases}
\bar{\phi}= t^2(  \bar{\omega}_0 dy_0+ \bar{\omega}_1 dy_1+\bar{\omega}_2 dy_2)+ \bar{\lambda},
\\
\bar{\psi}= h^{2/3} (t^4 \bar{\mu}-t^2 (\bar{\Theta}_0dy_1dy_2+ \bar{\Theta}_1dy_2dy_0+ \bar{\Theta}_2dy_0dy_1  ) ) .   
\end{cases}
\end{equation}
Here modulo lower order terms, $\bar{\omega}_i$ and $\bar{\Theta}_i$ are horizontal-vertical type $(0,2)$ forms, $\bar{\lambda}=-\lambda dy_0dy_1dy_2$ for $\lambda>0$, and $\bar{\mu}$ has horizontal-vertical type $(0,4)$. The parameter $0<t\ll 1$ controls the K3 collapsing rate. The coassociative condition $\bar{\phi}|_{K3}=0$ is encoded in the ansatz. The $G_2$-structure imposes a number of linear algebraic constraints:
\begin{equation}\label{hypersymplecticforms}
\begin{cases}
\bar{\mu}= \lambda^{-2/3}\det^{1/3} ( \frac{1}{2} \bar{\omega}_a\wedge \bar{\omega}_b ),
\\
\bar{\omega}_i \wedge \bar{\Theta}_j= 2\lambda^{1/3} \delta_{ij} \det^{1/3}( \frac{1}{2} \bar{\omega}_a\wedge \bar{\omega}_b )= 2\lambda \delta_{ij} \bar{\mu}.
\end{cases}
\end{equation}

This setting dimensionally reduces to the $G_2$-version of iterative collapsing picture in section \ref{fastcirclecollapsing} when the $\R$-variable $y_0$ splits off. Explicitly, the relation to the Apostolov-Salamon $SU(3)$-structure is given by
\[
\begin{cases}
\bar{\omega}_0= \omega_y,  \quad \bar{\Theta}_0= h^{1/3} \omega_y,
\\
\bar{\omega}_1 dy_1+\bar{\omega}_2 dy_2=-h^{1/4} \text{Im}\Omega,
\\
\bar{\lambda}=- \omega_S\wedge dy_0,
\\
\bar{\Theta}_2 dy_1-\bar{\Theta}_1 dy_2=- h^{1/4} \text{Re}\Omega, 
\\
\bar{\mu}=\frac{1}{2} h^{1/3} \omega_y^2,
\end{cases}
\]
so that 
\[
\begin{cases}
\Phi= dy_0\wedge (\epsilon \vartheta\wedge \omega+h^{3/4} \text{Re}\Omega)+ (- \epsilon h^{1/4}\vartheta\wedge \text{Im}\Omega+ \frac{1}{2}h\omega^2),
\\
\bar{\phi}= -\omega dy_0- h^{1/4} \text{Im}\Omega,
\\
\bar{\psi}=- h^{3/4} \text{Re}\Omega\wedge dy_0+ \frac{1}{2} h\omega^2.
\end{cases}
\]

Returning to the $Spin(7)$ story,
the fast circle collapsing case is $\epsilon \ll t^2\ll 1$ as before. To zeroth order, the circle bundle is invisible, and we only see the geometry of the 7-manifold with an approximately torsion free $G_2$-structure admitting a collapsing coassociative K3 fibration, described as in Donaldson's proposal \cite{Donaldson} by a maximal submanifold in $H^2(K3)$.

\subsection{Iterative fibration: fine tuned collapsing}

Now we consider the \textbf{fine tuned collapsing} $\epsilon=t^2$, and try to encode it into adiabatic data.

\begin{itemize}
\item 
Since $d\bar{\phi}=0$, the $\bar{\omega}_i$ restricted to K3 fibres are closed, so defines a `hypersymplectic triple'. The function $\lambda$ is to leading order constant on fibres. Since we also have $d\bar{\psi}+ \epsilon d\vartheta\wedge \bar{\phi}=0$, by imposing to leading order $h$ is constant on K3 fibres, we get that $\bar{\Theta}_i$ are approximately closed on fibres. Now $\bar{\Theta}_i$ and $\bar{\omega}_i$ are two bases for the anti-self-dual 2-forms on the K3 fibres, so they must be fibrewise related to each other by linear transformations. This shows to leading order the K3 fibres are \textbf{hyperk\"ahler}, even though $\bar{\omega}_i$ are not orthonormal in general. We write 
\[
\bar{g}_{ij}=\int_{K3} \bar{\omega}_i\wedge \bar{\omega}_j.
\]
By (\ref{hypersymplecticforms}), the linear transformations can be determined cohomologically as
\begin{equation}\label{Spin7Thetaiomegai}
\bar{\Theta}_i= 2\lambda \bar{g}^{ik} \bar{\omega}_k \int_{K3}\bar{\mu}.
\end{equation}

\item
From $d\bar{\phi}=0$ mod $dy_0\wedge dy_1\wedge dy_2$ terms, we have the integrability condition 
\[
\frac{\partial}{\partial y_i}[\bar{\omega}_j]= \frac{\partial}{\partial y_j}[\bar{\omega}_i],
\]
so over a local base there is a map $H: B\to H^2(K3)$, such that $[\bar{\omega}_i]=- \frac{\partial H}{\partial y_i}$. Here the minus sign is for the compatibility with our dim reduction to the $G_2$-case in section \ref{Spin7dimrectiontoG2}. The map $H$ is a \textbf{positive section} in the sense that the induced metric on $B$ is Riemannian.

\item
From $d\bar{\psi}+ \epsilon d\vartheta\wedge \bar{\phi}=0$ mod $dy_i\wedge dy_j$ terms, we have 
\[
\frac{\partial}{\partial y_i} [h^{2/3}\bar{\mu}]= - [d\vartheta]\cdot [\bar{\omega}_i]=  \frac{\partial}{\partial y_i}( [d\vartheta]\cdot H  )\in H^4(K3).
\]
Thus \[
h^{2/3}[\bar{\mu}]= [d\vartheta]\cdot H+ \text{const}\]
Up to adjusting $H$ by an additive constant, we can arrange 
\begin{equation}\label{Spin7volumenormalisation}
h^{2/3}[\bar{\mu}]=\begin{cases}
[d\vartheta]\cdot H,\quad & [d\vartheta]\neq 0,
\\
\text{const}, \quad & [d\vartheta]=0.
\end{cases} 
\end{equation}

\item

Since $d\bar{\phi}=0$, the 7-dimensional component of $d*_{\bar{\phi}}\bar{\phi}$ vanishes, \ie
\[
d(\psi h^{-2/3})\wedge \iota_X \bar{\phi}=0, \quad \forall X .
\]
Using $d\psi= -\epsilon d\vartheta\wedge \bar{\phi}$, 
\[
(\epsilon h d\vartheta \wedge \bar{\phi}+ \frac{2}{3} \bar{\psi}\wedge dh) \wedge \iota_X \bar{\phi}=0.
\]
Substituting $X$ the horizontal lift of $\frac{\partial}{\partial y_i}$, and the component formulae (\ref{Spin7structurecomponents}) for the $Spin(7)$ structure, and noticing $\bar{\omega}_i\wedge \bar{\Theta}_i= 2\lambda \bar{\mu}$, we derive
\begin{equation}\label{Spin7harmonicform}
h^{1/3} d\vartheta \wedge \omega_i = \bar{\mu} \frac{\partial h}{\partial y_i}.
\end{equation}
Now the formula $\bar{\mu}= \det^{1/3}(\bar{\omega}_a\wedge \bar{\omega}_b) \lambda^{-2/3}$ indicates that $\bar{\mu}$ is up to a fibrewise constant the hyperk\"ahler volume form on the K3, so the above equation implies $d\vartheta$ restricted to the K3 fibres are the \textbf{harmonic 2-forms} in $[d\vartheta]$ up to the leading order. Morever, taking the cohomology classes of (\ref{Spin7harmonicform}), and comparing with (\ref{Spin7volumenormalisation}), we get
\[
\frac{\partial}{\partial y_i}(h^{-1/3}[\bar{\mu}])=0,
\]
or equivalently
\begin{equation*}
h^{-1/3}[\bar{\mu}]=\text{const}.
\end{equation*}
The value of this constant can be prescribed, up to a global scaling of the $Spin(7)$ 4-form. The normalisation consistent with our $G_2$ story is 
\begin{equation}
\int_{K3} \bar{\mu}= \frac{1}{2} h^{1/3}.
\end{equation}
(In that dim reduction, this normalisation corresponds to $\int_{K3}[\omega_y]^2=1$.) Substituting back into (\ref{Spin7volumenormalisation}), we recover
\begin{equation}\label{hviaHSpin7}
h=\begin{cases}
2[d\vartheta]\cdot H,\quad & [d\vartheta]\neq 0,
\\
\text{const}, \quad & [d\vartheta]=0.
\end{cases} 
\end{equation}
Notice the positivity of $h$ is an extra a priori requirement on $H$.

\item
From the hyperk\"ahler K3 condition,
\[
\bar{\omega}_i\wedge \bar{\omega}_j= Q_{ij} \bar{\mu},
\]
for some fibrewise constant matrix $Q_{ij}$ to leading order. The coefficient $Q_{ij}$ is determined by integrating over the K3, so that
\[
Q_{ij}= \bar{g}_{ij} (\int_{K3} \bar{\mu})^{-1}.
\]
Thus 
\[
\text{det} ^{1/3}( \frac{\bar{\omega}_a\wedge \bar{\omega}_b}{2} )= \frac{1}{2} \text{det}^{1/3}(Q) \bar{\mu}= \frac{ \det^{1/3}(\bar{g}) }{ 2\int \bar{\mu} } \bar{\mu}.
\]
Comparing with $\bar{\mu}= \text{det} ^{1/3}( \frac{\bar{\omega}_a\wedge \bar{\omega}_b}{2} ) \lambda^{-2/3}$, we have
\[
\lambda^{2/3} = \frac{ \det^{1/3}(\bar{g}) }{ 2\int \bar{\mu} }= \frac{ \det^{1/3}(\bar{g}) }{ h^{1/3} }
\]
hence
\begin{equation}
\lambda= \text{det}^{1/2} (\bar{g}) h^{-1/2}.
\end{equation}
Now our previous formula (\ref{Spin7Thetaiomegai}) simplifies to
\begin{equation}
\Theta_i= \lambda h^{1/3} \bar{g}^{ik} \bar{\omega}_k=\text{det}^{1/2} (\bar{g}) h^{-1/6} \bar{g}^{ik} \bar{\omega}_k.
\end{equation}
By this stage we have successfully expressed all quantities to leading order in terms of $H$.

\item
By looking at the horizontal-vertical 
(3,2) component
of $d\bar{\psi}+ \epsilon d\vartheta\wedge \bar{\phi}=0$,
we obtain to leading order
\[
\sum_i \partial_i (h^{2/3} [\bar{\Theta}_i]) + \lambda [d\vartheta]=0.
\]
Substituting in the formulae for $\bar{\Theta}_i$ and $\lambda$, we obtain
\begin{equation}
\sum_i \partial_i (  h^{1/2}\text{det}^{1/2} (\bar{g}) \bar{g}^{ik} \partial_k H) = \text{det}^{1/2} (\bar{g}) h^{-1/2} [d\vartheta] ,
\end{equation}
where $h$ is related to $H$ by (\ref{hviaHSpin7}). This is the \textbf{weighted maximal submanifold equation} we saw previously, except now over a 3-dimensional manifold. For instance, the variational formulation in section \ref{Variationalweightedmaximal} holds verbatim in this 3-dimensional case.

\end{itemize}

Morever, given a solution of the weighted maximal submanifold equation over a local base $B$, a procedure closely analogous to section \ref{Slowcollapsing} reconstructs an approximately torsion free $\Phi$, so that the above formulae hold to leading order. In the special case $[d\vartheta]=0$, this agrees with Donaldson's proposal \cite{Donaldson}.

\subsection{Global discussions}

The local compactification mechanisms in the $G_2$-case have natural analogues in the $Spin(7)$ case:

\begin{itemize}
\item
The analogue of the Lefschetz fibration is the Kovalev-Lefschetz singularity disucssed in detail in Donaldson's proposal \cite{Donaldson}.
This is basically a parametrized version of the Lefschetz singularity along a curve in the 7-fold, projecting down to a knot inside the 3-dimensional base $B$. The Picard-Lefschetz monodromy can be encoded into affine orbifold bundles.

\item
The wall crossing phenomenon has the same transverse behaviour, the difference being that the walls are now surfaces. The walls are still characterized by $H\cdot \sigma=\text{const}$ where $\sigma\in H^2(K3)$ is the jumping of the first Chern class. This condition is intimately related to adiabatic coassociative submanifolds in collapsing $G_2$-manifolds, just as our previous wall crossing story relates to adiabatic special Lagrangians. In the simplest case, the walls are disjoint, but the knots in $B$ coming from the projection of the Lefschetz singular locus should be allowed to intersect the walls in $B$. Inside the 7-fold, the curve of Lefschetz singularities generically stays disjoint from the 4-dimensional discriminant locus associated with wall crossing, even if their projections in $B$ intersect.

\item
Similarly, the boundary orientifold behaviour generalizes to the $Spin(7)$ case.

\end{itemize}

Thus the global formulation of the weighted maximal submanifold equation makes sense for the $Spin(7)$ case, just like in the $G_2$ case described in section \ref{Globaltopology}. In the special case of $[d\vartheta]=0$ without any walls or boundaries, it recovers Donaldson's proposal. The topological situation is now much richer than the 2-dimensional case. It involves a local system with singularities along knots, over a 3-manifold possibly with boundary; the interplay between the $Spin(7)$-geometry and the 3-manifold topology seems well worth exploring.


\begin{thebibliography}{7}
	


\bibitem{ApostolovSalamon} 



Apostolov, Vestislav; Salamon, Simon. Kähler reduction of metrics with holonomy $G_2$. Comm. Math. Phys. 246 (2004), no. 1, 43--61.





\bibitem{BakryEmery}
 Ambrosio, Luigi; Gigli, Nicola; Savaré, Giuseppe. Bakry-Émery curvature-dimension condition and Riemannian Ricci curvature bounds. Ann. Probab. 43 (2015), no. 1, 339--404.
 
 
 \bibitem{AtiyahWitten} 
 
 Atiyah, Michael; Witten, Edward. $M$-theory dynamics on a manifold of $G_2$ holonomy. Adv. Theor. Math. Phys. 6 (2002), no. 1, 1--106.
 
 
 
\bibitem{Barbosa} 
Barbosa,
Rodrigo. A Deformation Family for Closed G2-Structures on ADE Fibrations. 	arXiv:1910.10742.

\bibitem{Hubner1} 

Andreas P. Braun, Sebastjan Cizel, Max H\"ubner, Sakura Schafer-Nameki.

Higgs Bundles for M-theory on $G_2$-Manifolds. 	arXiv:1812.06072.




\bibitem{Chen} 
Gao Chen, G2 manifolds with nodal singularities along circles, accepted by the Journal of Geometric Analysis.   
   
   
\bibitem{HaskinsCorti} 

Corti, Alessio; Haskins, Mark; Nordström, Johannes; Pacini, Tommaso. $G_2$-manifolds and associative submanifolds via semi-Fano 3-folds. Duke Math. J. 164 (2015), no. 10, 1971--2092. 




\bibitem{Donaldson} 

Donaldson, Simon. Adiabatic limits of co-associative Kovalev-Lefschetz fibrations. Algebra, geometry, and physics in the 21st century, 1--29, Progr. Math., 324, Birkhäuser/Springer, Cham, 2017.



\bibitem{Foscolo}  

Foscolo, Lorenzo. ALF gravitational instantons and collapsing Ricci-flat metrics on the $K3$ surface. J. Differential Geom. 112 (2019), no. 1, 79--120.


\bibitem{HaskinsS1} 
Lorenzo Foscolo, Mark Haskins, Johannes Nordström.Complete non-compact G2-manifolds from asymptotically conical Calabi-Yau 3-folds. arXiv:1709.04904


\bibitem{HeinSun} 
Hein, Hans-Joachim; Sun, Song. Calabi-Yau manifolds with isolated conical singularities. Publ. Math. Inst. Hautes Études Sci. 126 (2017), 73--130.



\bibitem{HSVZ} 

Hans-Joachim Hein, Song Sun, Jeff Viaclovsky, Ruobing Zhang.
Nilpotent structures and collapsing Ricci-flat metrics on K3 surfaces. 	arXiv:1807.09367




\bibitem{Hubner} 

H\"ubner, Max.
Local G2-Manifolds, Higgs Bundles and a Colored Quantum Mechanics. arXiv:2009.07136.


\bibitem{JoyceKarigiannis} 


Dominic Joyce, Spiro Karigiannis. A new construction of compact torsion-free $G_2$-manifolds by gluing families of Eguchi-Hanson spaces. To appear in Journal of Differential Geometry.


\bibitem{Kovalev} 

Kovalev, Alexei. Twisted connected sums and special Riemannian holonomy. J. Reine Angew. Math. 565 (2003), 125--160. 



\bibitem{KovalevLee} 
Kovalev, Alexei; Lee, Nam-Hoon. $K3$ surfaces with non-symplectic involution and compact irreducible $G_2$-manifolds. Math. Proc. Cambridge Philos. Soc. 151 (2011), no. 2, 193--218.



\bibitem{Limaximal} 

Li, Yang.
Dirichlet problem for maximal graphs of higher codimension. 	arXiv:1807.11795. International Mathematics Research Notices.


\bibitem{LiCY} 
Li, Yang. A new complete Calabi-Yau metric on $\C^3$. Invent. Math. 217 (2019), no. 1, 1--34.



\bibitem{LiCY2} 
Li, Yang. On collapsing Calabi-Yau fibrations. Accepted by Journal of Differential Geometry.


\bibitem{LiCY3} 
Li, Yang. A gluing construction of collapsing Calabi-Yau metrics on K3 fibred 3-folds. Geom. Funct. Anal. 29 (2019), no. 4, 1002--1047.


\bibitem{Pantev} 
Pantev, Tony; Wijnholt, Martijn. Hitchin's equations and M-theory phenomenology. J. Geom. Phys. 61 (2011), no. 7, 1223--1247.



\bibitem{SunZhang} 

Sun, Song; Zhang, Ruobing. Complex structure degenerations and collapsing of Calabi-Yau metrics. 	arXiv:1906.03368



\end{thebibliography}
\end{document}